\documentclass[11pt]{amsart}

\usepackage{setspace} 
\usepackage{amssymb}
\usepackage[dvips]{color}
\usepackage{amsmath}
\usepackage{amsthm}
\usepackage[a4paper]{geometry}
\usepackage{amsfonts}
\usepackage[all]{xypic} 
\usepackage{bbm}
\usepackage{xcolor}

\usepackage{color}
\usepackage[colorlinks=true,linkcolor=blue,citecolor=blue,pdfpagelabels=false]{hyperref}
\usepackage{cleveref}
\usepackage{url}

\usepackage{hyperref}

\usepackage[colorinlistoftodos]{todonotes}

\allowdisplaybreaks[4] 
\newtheoremstyle{myremark}     {10pt}{10pt}{}{}{\bfseries}{.}{.5em}{}

\newtheorem{thm}{Theorem}[section]
\newtheorem{cor}[thm]{Corollary}
\newtheorem{lem}[thm]{Lemma}
\newtheorem{prop}[thm]{Proposition}
\theoremstyle{definition}
\newtheorem{defn}[thm]{Definition}

\theoremstyle{myremark}
\newtheorem{rem}[thm]{Remark}
\numberwithin{equation}{section}

\def\C{\mathbb C}
\def\L{\mathbb L}
\def\E{\mathbb E}

\def\I{\mathbb I}
\def\N{\mathbb N}

\def\R{\mathbb R}
\def\T{\mathbb T}
\def\Z{\mathbb Z}

\newcommand{\CB}{\mathcal{B}}
\newcommand{\CC}{\mathcal{C}}
\newcommand{\CD}{\mathcal{D}}
\newcommand{\CE}{\mathcal{E}}

\newcommand{\CH}{\mathcal{H}}

\newcommand{\h}{\mathcal{H}}

\newcommand{\CM}{\mathcal{M}}
\newcommand{\CN}{\mathcal{N}}

\newcommand{\CP}{\mathcal{P}}

\newcommand{\CT}{\mathcal{T}}

\newcommand{\CW}{\mathcal{W}}

\newcommand{\eps}{\varepsilon}
\newcommand{\ups}{\upsilon}

\newcommand{\abs}[1]{\left\vert#1\right\vert}

\newcommand{\norm}[1]{\left\Vert#1\right\Vert}

\newcommand\inner[2]{\left\langle #1, #2 \right\rangle}

\newcommand{\ri}{\rightarrow}
\newcommand{\limn}{ \underset{n \ri \infty}{\lim} }

\begin{document}
	
	\title[Maximal Ergodic Inequality]{Maximal Inequality Associated to Doubling Condition for State Preserving Actions }
	
	\author[P. Bikram]{Panchugopal Bikram}
	\address{School of Mathematical Sciences,
National Institute of Science Education and Research,  Bhubaneswar, An OCC of Homi Bhabha National Institute,  Jatni- 752050, India}
	\email{bikram@niser.ac.in}

	\author[D. Saha]{Diptesh Saha}
	\address{Theoretical Statistics and Mathematics Unit, Indian Statistical Institute, Delhi, India}
	\email{dptshs@gmail.com}

	

	
	

	\begin{abstract}
		In this article, we prove maximal inequality and ergodic theorems for state preserving actions on von Neumann algebra by an amenable, locally compact, second countable group equipped with the metric satisfying the doubling condition.  
	The key idea is to use Hardy–Littlewood maximal inequality,  a version of the transference principle and certain norm estimates of differences between ergodic averages and martingales. 
	\end{abstract}

	\maketitle 

	\section{Introduction}
	
	In this article, we aim to study non-commutative maximal ergodic inequalities and pointwise ergodic theorems for certain group actions. The origin of ergodic theorems on classical measure spaces traces back to the 1930s due to Birkhoff and von Neumann.  The subject has evolved a lot since then. The non-commutative analogue of the ergodic theorems first appeared in the works of Lance (\cite{Lance1976}) in 1976. In (\cite{Lance1976}), the author proved a pointwise ergodic theorem for the average of a single automorphism preserving a faithful, normal state in von Neumann algebra. These results are then generalised by Kummerer, Conze, Dang-Ngoc and many others (cf. \cite{Kuemmerer1978}, \cite{Conze1978} and references therein).
Eventually, it became evident that the standard approach to establishing an ergodic theorem is to show a maximum inequality and the accompanying Banach principle. 
However, establishing maximal inequality is one of the hardest and most delicate part in the non-commutative setting.

	In the seminal works \cite{Yeadon1977} and \cite{Yeadon1980}, Yeadon obtained maximal ergodic theorems for actions of positive sub-tracial, sub-unital maps on the preduals of semifinite von Neumann algebras and proved associated pointwise ergodic theorems. Later on, Junge and Xu (\cite{Junge2007})) extended Yeadon's maximal inequalities to prove Dunford-Schwartz maximal inequalities in the non-commutative $L^p$-spaces using interpolation techniques.
	
	On the other hand, various results are obtained regarding the ergodic theorems for actions of various groups on classical measure spaces (cf. \cite{anantharaman2010ergodic}, \cite{calderon1953general}). Among them, \cite{Lindenstrauss2001} is particularly important for our purpose. He proved pointwise ergodic theorems for tempered F\o lner sequences, which is considered as the generalisation of Birkhoff's ergodic theorem for actions of second countable, amenable groups. In the non-commutative realm, the first significant breakthrough is made by Hong, Liao and Wang in \cite{Hong2021}. In this article, the authors obtained several maximal inequalities on non-commutative $L^p$-spaces for the actions of locally compact group metric measure spaces satisfying doubling conditions and compactly generated groups of polynomial growth and used them to prove associated pointwise ergodic theorems. Furthermore, these results are extended to the general setting of amenable groups in \cite{cadilhac2022noncommutative}. It is worth mentioning here that both these articles deal with actions preserving a faithful, normal, semifinite trace. However, in \cite{Junge2007}, the authors studied maximal inequalities for ergodic averages in Haagerup's $L^p$-spaces ($1< p < \infty$) associated to state preserving actions of the group $\Z$ and semigroup $\R_+$. However,  the maximal inequality for state preserving actions in the endpoint case ($p=1$) was not studied until \cite{Bik-Dip-neveu}, \cite{pg-ds-ergo-semi}.
	
	In \cite{Bik-Dip-neveu}, the authors studied maximal inequalities and pointwise ergodic theorems on non-commutative $L^1$-spaces for locally compact polynomial growth group actions preserving a faithful, normal state on a von Neumann algebra. To obtain the aforesaid maximal inequality, we basically used an inequality obtained in \cite[Proposition 4.8]{Hong2021} to dominate the associated ergodic averages for groups by an ergodic average of Markov operator.
 In \cite{pg-ds-ergo-semi}, the authors have established similar maximal inequality for the action of $\Z^d_+$ and $\R^d_+$ by using a similar method. In this setup, obtaining maximal inequality for actions of general amenable groups is known to be challenging in the literature, and it requires delicate analysis. However, in this article, we obtain maximal inequalities for actions of group metric measure spaces satisfying the doubling condition using suitable analogues of Calderon's transference principle and maximal inequality for Martingales.
	
	Let $G$ be an amenable, locally compact, second countable group equipped with a right invariant Haar measure $\mu$. We further assume $G$ is a metric space that comes with an invariant metric (cf. Eq. \ref{inv metric-defn}) $d$ satisfying a doubling condition (cf. Eq. \ref{doub cond-defn}). Suppose there exists an increasing sequence of real numbers $(r_n)_{n \in \N}$ such that the balls $V_{r_n}:=\{s \in G: d(e,s)<r_n\}$ satisfy for all $i \in \N$  
		\begin{enumerate}
			\item[(i)] $2^i \leq r_i < 2^{i+1}$ and 
			\item[(ii)] $\frac{\mu(V(x, r_i) \setminus V(x, r_i- (3/2)^i))}{\mu(V(x, r_i))} \leq C' (3/4)^i$ for all $x \in G$,
		\end{enumerate}
  where $C'$ is a constant. Let $\CM$ be a von Neumann algebra and consider the non-commutative dynamical system $(\CM, G, \alpha)$ (cf. Definition \ref{nc dyn sys}). Further suppose that $\CM$ has a faithful, normal state $\rho$ such that $\rho \circ \alpha_s = \rho$ for all $s \in G$. Now one of the main results of this article states the following maximal inequality and ergodic theorem.
	
	\begin{thm}\label{main thm}
		Let $\phi \in \CM_*$. Consider the averages $(A_k(\phi))_{k \in \N}$, where
		\begin{align*}
			A_k(\phi)(x):= A_{r_k}(\phi)(x)= \frac{1}{\mu(V_{r_k})} \int_{V_{r_k}} \phi(\alpha_s(x)) d\mu(s), ~ x \in \CM.
		\end{align*}
		Then,
		\begin{enumerate}
			\item $(A_k(\phi))_{k \geq 1}$ has the property: there exists $\kappa \in \N$ such that for all $n \in \N$ and $\eps, \delta, \upsilon>0$, there exists projections $\{e^{i, \delta, \ups}_{n, \eps}:  i \in [\kappa] ~\}$ such that for all $ 1 \leq i  \leq \kappa $
			
			\begin{enumerate}
				\item $\rho(1- e^{i, \delta, \ups}_{n, \eps}) < \frac{1+ \ups }{\eps}   \norm{  \phi} +\ups $
				\item $A_k(\phi)(x) \leq C\delta \norm{\phi} \norm{x} +  C\eps \rho(x)$ for all $x \in e^{i, \delta, \ups}_{n, \eps} \CM_+ e^{i, \delta, \ups}_{n, \eps}$ and $ 1 \leq k \leq n$.
			\end{enumerate}
			The constant $C$ only depends on the group $G$.\\
			
			\item Furthermore, if $\CM$ is equipped with a faithful, normal tracial state $\tau$, then there exists $\bar{\phi} \in \CM_*$ such that for all $\epsilon>0$ there exists a projection $e \in \CM$ with $\tau(1-e)< \epsilon$ and 
			\begin{align*}
				\lim_{k \to \infty} \sup_{x \in eM_+e, x \neq 0} \frac{\abs{(A_k(\phi) - \bar{\phi})(x)}}{\tau(x)}=0,
			\end{align*}
i.e, $A_n(\phi)$ converges b.a.u. to $\bar{\phi}$ (cf. Definition \ref{bau conv- functionals}).
\end{enumerate}
	\end{thm}
 We note that compared to our previous maximal inequality obtained in \cite{Bik-Dip-neveu} for the polynomial growth group, this maximal inequality is much more intricate due to the dependency of the projection on more variables and for the error term ($C \delta \norm{ \phi } \norm{ x}  $   \ref{main thm}(b)). 

Now, we describe the layout of the article. In the next section, we recollect some basics of von Neumann algebra and the linear functionals defined in it. Further, we define non-commutative $L^1$-spaces and discuss b.a.u topology and end the section with the definition of non-commutative dynamical systems and predual actions. \S \ref{Weak type inequality} contains the proof of part $(1)$ of Theorem \ref{main thm}. \S \ref{Pointwise ergodic theorems} is devoted to the proof of pointwise ergodic theorem associated with the ball averages corresponding to a lacunary sequence, which proves part $(2)$ of Theorem \ref{main thm}. We also deduce an appropriate Banach principle in this setup. In the final section, we collect the results from \S \ref{Weak type inequality} and \S \ref{Pointwise ergodic theorems} and combine them with the Neveu decomposition result from \cite{Bik-Dip-neveu} to obtain a Stochastic ergodic theorem in this setup.

	\section{Preliminaries}
	Throughout this article, $\CM$ will denote a von Neumann algebra represented in a standard form on a separable Hilbert space $\CH$.  The norm on $\CM$ induced from $\CB(\CH)$ will be denoted by $\norm{\cdot}$. $\CM'$ will denote the commutant of $\CM$. We will denote by $\CP(\CM)$, the lattice of projections in the von Neumann algebra $\CM$ and $\CP_0(\CM):= \CP(\CM) \setminus \{0\}$.

	The predual of the von Neumann algebra is a closed linear subspace of $\CM^*$, the Banach dual of $\CM$, and is denoted by $\CM_*$. The elements of $\CM_*$ are actually the $w$-continuous linear functionals on $\CM$. Moreover, the norm on $\CM_*$ will be denoted by $\norm{\cdot}_1$. A linear functional $\phi \in \CM^*$ is called positive if $\phi(x^*x)\geq 0$ for all $x \in \CM$, and $\phi$ is called self-adjoint if $\phi(x)=\overline{\phi(x^*)}$ for all $x \in \CM$. The positive and self-adjoint elements of $\CM_*$ will be denoted by $\CM_{*+}$ and $\CM_{*s}$ respectively.
	
	\subsection{\textbf{Non-commutative $L^1$-spaces}}
	Let $\CM \subseteq \CB(\CH)$ be a von Neumann algebra equipped with a faithful, normal, semifinite (f.n.s.) trace $\tau$. A ( possibly unbounded ) operator $X: \CD(X) \subseteq \CH \to \CH$ which is affiliated to $\CM$ ( that is, $u'^* X u'= X$ for all unitary operator $u'$ in $\CM'$), closed and densely defined, will be denoted by $X \eta \CM$. Now for $X \eta \CM$, it is called $\tau$-measurable if for all $\epsilon>0$ there exists $e \in \CP(\CM)$ such that $\tau(1-e) < \epsilon$ and $e \CH \subseteq \CD(X)$ and one defines the set
	\begin{align*}
		L^0(\CM, \tau):= \{X \eta \CM: X \text{ is } \tau\text{-measurable}\}.
	\end{align*}
	
	Notice that $L^0(\CM, \tau)$ is a $*$-algebra with respect to adjoint operation, strong sum ( that is $X+Y:= \overline{X+Y}$ ) and strong product ( that is $X \cdot Y:= \overline{X \cdot Y}$ ), where $\overline{X}$ denote the closure of the operator $X$ for all $X \in L^0(\CM, \tau)$. For more information, we refer to \cite{stratila2019lectures}, \cite{hiai2021lectures}.
	
	There are several topologies on $L^0(\CM, \tau)$. Two of them, which are relevant to this article, are discussed in the following. The measure topology is given by the neighbourhood system $\{X + \CN(\epsilon, \delta): X \in L^0(\CM, \tau), \epsilon, \delta>0\}$, where
	\begin{align*}
		\CN(\epsilon, \delta):= \{X \in L^0(\CM, \tau): \exists~ e \in \CP(\CM) \text{ such that } \tau(1-e)< \delta, \text{ and } \norm{eXe} < \epsilon\}.
	\end{align*}
	
	Clearly, a net $\{X_i\}_{i \in I}$ in $L^0(\CM, \tau)$ converges in measure topology to $X \in L^0(\CM, \tau)$ if for all $\epsilon, \delta>0$ there exists $i_0 \in I$ such that for all $i \geq i_0$ there exists $e_i \in \CP(\CM)$ with $\tau(1-e_i)< \delta$ and
	\begin{align*}
		\norm{e_i(X_i - X)e_i}< \epsilon.
	\end{align*} 
	
	Note that it follows from \cite[Theorem 1]{nelson1974notes}, that $L^0(\CM, \tau)$ is complete with respect to the measure topology. Moreover, $L^0(\CM, \tau)$ becomes a metrizable Hausdorff space with respect to this topology in which $\CM$ is dense. On the other hand, we have the following definition for bilateral almost uniform topology.
	
	\begin{defn}\label{bau conv defn}
		A net $\{X_i\}_{i \in I}$ in $L^0(\CM, \tau)$ converges in bilateral almost uniform (b.a.u.) topology to $X \in L^0(\CM, \tau)$ if for all $\epsilon>0$ there exists $e \in \CP(\CM)$ with $\tau(1-e)< \delta$  such that 
		\begin{align*}
			\lim_i \norm{e(X_i - X)e}=0.
		\end{align*} 
	\end{defn}
	
	We have the following result from \cite[Theorem 2.2]{litvinov2024notes}, which is important to our purpose.
	
	\begin{thm}\label{complete wrt bau}
		The space $L^0(\CM, \tau)$ is complete with respect to b.a.u. topology.
	\end{thm}
	
	Let us also observe the following.
	
	\begin{prop}
		Suppose $\{X_i\}_{i \in \I}$ and $\{Y_i\}_{i \in \I}$ are two nets in $L^0(\CM, \tau)$ such that $\{X_i\}_{i \in \I}$ converges in measure (resp. b.a.u.) to $X$ and $\{Y_i\}_{i \in \I}$ converges in measure (resp. b.a.u.) to $Y$ . Then, for all $c \in \C$, $\{cX_i +Y_i\}_{i \in I}$ converges in measure (resp. b.a.u.) to $cX + Y$.
	\end{prop}
	
	The trace $\tau$ on $M_+$ can be extended to $L^0(\CM, \tau)_+$ via
	\begin{align*}
		\tau(X):= \int_0^\infty \lambda d \tau(e_\lambda),
	\end{align*}
	where $X= \int_0^\infty \lambda d e_\lambda$ is the spectral decomposition of $X \in L^0(\CM, \tau)_+$. The non-commutative $L^1$-space associated to $(\CM, \tau)$ is defined as 
	\begin{align*}
		L^1(\CM, \tau):= \{X \in L^0(\CM, \tau): \tau(\abs{X})< \infty\},
	\end{align*}
	and $\norm{\cdot}_1:= \tau(\abs{\cdot})$ defines a norm on $L^1(\CM, \tau)$. Now, we state a few properties of $L^1(\CM, \tau)$, which will be recursively used in the subsequent sections. For proof of the facts, we refer to \cite[Chapter 4]{hiai2021lectures}.
	
	\begin{thm}\label{predual prop}
		\begin{enumerate}
			\item $L^1(\CM, \tau)$ is a Banach space with respect to the norm $\norm{\cdot}_1$.
			\item For $X \in L^1(\CM, \tau)$ and $y \in \CM$, $Xy,~yX \in L^1(\CM, \tau)$ and $\tau(Xy)= \tau(yX)$.
			\item We have $\CM_*= L^1(\CM, \tau)$ with respect to the correspondence $\phi \in \CM_* \leftrightarrow X \in L^1(\CM, \tau)$ given by $\phi(x)= \tau(Xx)$ for all $x \in \CM$. Furthermore, the above correspondence is a surjective linear isometry, which preserves the positive cone.
		\end{enumerate}
	\end{thm}
	
	\subsection{Non-commutative dynamical systems}
	Throughout this article, $G$ is assumed to be an amenable, locally compact, second countable, Hausdorff group with right invariant Haar measure $\mu$. Let $(E, \norm{\cdot})$ be a real ordered Banach space.
	
	\begin{defn}\label{nc dyn sys}
		A non-commutative dynamical system is a triple $(E, G, \gamma)$, where $\gamma$ is a map from $G$ to $\CB(E)$, the space of bounded operators on $E$, satisfying $\gamma_s \circ \gamma_t= \gamma_{st}$ for all $s,t \in G$ and 
		\begin{enumerate}
			\item for all $a \in E$, the map $G \ni s \to \gamma_s(a) \in E$ is continuous, here we take norm topology on $E$ and $w^*$-topology when $E= \CM$,
			\item $\sup_{s \in G} \norm{\gamma_s} < \infty$, and, 
			\item $\gamma_s(a) \geq 0$ for all $a \geq 0$ and $s \in G$.
		\end{enumerate}
	\end{defn}
	
	Let $\CM$ be a von Neumann algebra with f.n.s trace $\tau$. Consider the non-commutative dynamical system $(L^1(\CM, \tau), G, \gamma)$. Since $\CM= (L^1(\CM, \tau)^*$, for every $s \in G$, we recall the dual of $\gamma_s$, denoted by $\gamma_s^*$, is defined by
	\begin{align}\label{dual trans}
		\tau(\gamma_s^*(x) Y)= \tau(x \gamma_s(Y)), ~x \in \CM, ~ Y \in L^1(\CM, \tau).
	\end{align}
	
	Furthermore, for the non-commutative dynamical system $(\CM, G, \beta)$, for all $s \in G$, the predual transformation of $\beta_s$, denoted by $\hat{\beta_s}$, is defined by 
	\begin{align}\label{predual trans}
		\tau(\beta_s(x) Y)= \tau(x \hat{\beta_s}(Y)), ~x \in \CM, ~ Y \in L^1(\CM, \tau).
	\end{align}
	Observe that in this case, we have $(\hat{\beta_s})^*= \beta_s$ for all $s \in G$.  Sometimes, we write the same notation for $\gamma_s$, and it's the predual map, and if we write $ \gamma_s(\phi)$ for $ \phi \in \CM_*$, then in this case, $\gamma_s $ is understood to be the predual map.

	\section{ \textbf{Weak type maximal  inequality} }\label{Weak type inequality}
	
	In this section, we will prove a weak type maximal inequalities for certain predual actions. Let $(\CM, G, \alpha)$ be a non-commutative dynamical system. In \cite{Bik-Dip-neveu}, the authors established a weak type maximal inequality for the ergodic averages associated with a state preserving action of a group of polynomial growth with a compact, symmetric generating set (\cite[see Theorem 4.13]{Bik-Dip-neveu}). 
	
In this section, we prove the maximal inequality when the group $G$ is  equipped with a metric $d$ (in this case, $d$ is also assumed to be a measurable function on $G \times G$) satisfying the doubling condition. We make use of appropriate Calderon's transference principle and a Hardy-Littlewood maximal inequality to obtain our result (cf. Theorem \ref{maximal ineq for averages}).

	Let $\CM$ be a von Neumann algebra with a f.n positive linear functional $\rho$  on $\CM$. Assume that $T: \CM \to \CM$ is a $w$-continuous positive linear map satisfying
	\begin{enumerate}
		\item $T(1)\leq 1$, and
		\item $\rho \circ T= \rho$.
	\end{enumerate}
	The associated dual map $T^*: \CM^* \to \CM^*$ keeps the predual $\CM_*$ invariant. Therefore, for every $n \in \N$, we can define a positive linear map $S_n : \CM_* \to \CM_*$ by
	\begin{equation}
		S_n(\nu):= \frac{1}{n} \sum_{k=0}^{n-1} (T^*)^k (\nu), ~ \nu \in \CM_*.
	\end{equation}
	
	\noindent For simplicity, we introduce the following notations. For $ m \in \N$, we write 
	\begin{enumerate}
		\item $ [m]:= \{ 1, 2, \cdots, m \} \subset \N$, and 
		\item $[-m \mathrel{{.}\,{.}}\nobreak m]:= \{ -m, -(m-1), \cdots, 0,  \cdots m-1, m \} \subset \Z$.
	\end{enumerate}
	Now, we recall the following result from \cite[Theorem 4.2]{Bik-Dip-neveu}, which we will use recursively in this section; here we include a slick proof. 
	
	\begin{thm}\label{maximal lemma-1}
		Let $\nu \in \CM_{*+}$ and also suppose that $\epsilon>0$ and $n \in \N$. Then there exists  $e_{n, \eps} \in \CP(\CM)$ such that
		\begin{enumerate}
			\item $\rho(1-e_{n, \eps}) < \frac{\norm{\nu}}{\epsilon }$ and
			\item $S_k(\nu)(x) \leq \epsilon \rho(x)$ for all $x \in (e_{n, \eps} \CM e_{n, \eps})_+$ and $k \in [n].$
		\end{enumerate}	
	\end{thm}

\begin{proof}
    Consider the $w$-compact subset $\mathcal{L}$ of the von Neumann algebra $R:= \displaystyle \bigoplus_{k=1}^{n} M$ defined by
  	\begin{align*}
  	\mathcal{L}:= \Big\{ \underbar{x}:= (x_1,\ldots, x_n) \in R : x_i\geq 0 \  \text{ and } \sum_{i=1}^{n} x_i \leq 1 \Big\}.
  	\end{align*}

   Now consider the ultraweakly ( $\sigma$-weakly ) continuous function  $\kappa$ on $R$ defined by
  	\begin{align*}
  	\kappa(\underbar{x}):= \sum_{k=1}^{n}  k\Big[ S_k(\nu')(x_{k}) - \rho(x_{k}) \Big], \text{ for all } \underbar{x} \in R,
  	\end{align*}
   where $\nu':= \nu/\epsilon$. Since $\mathcal{L}$ is $\sigma$-weakly compact,  the supremum  value of the function $\kappa$ on the set $\mathcal{L}$ will be attained. Let $\underbar{x} \in \mathcal{L}$ be such that $\kappa(\underbar{x}) \geq \kappa (\underbar{a})$ for all $\underbar{a} \in \mathcal{L}$.

We define $c:= 1- \sum_{k=1}^{n} x_i $. Note that $c \geq 0$. Let $c'$ be an element in $M$ such that $0 \leq c' \leq c$. For $m \in \{ 1, \ldots ,n \}$, take $\underline{x}' = (x_1, \ldots, x_m+c', \ldots, x_n) $. Then note that $\underbar{x}' \in \mathcal{L}$, 
  	Then we have,
  	\begin{align*}
  	\kappa (\underbar{x}) \geq \kappa (\underbar{x}').
  	\end{align*}

Therefore, we deduce the following
  	\begin{align}\label{a}
  	\nonumber&  \kappa (\underbar{x}) - \kappa (\underbar{x}') \geq 0 \\
  	\nonumber \Rightarrow \qquad &\sum_{k=1}^{n} k\Big[ S_k(\nu')(x_{k}) - \rho(x_{k}) \Big] \\
  	\nonumber   - &\sum_{k=1}^{n} k \Big[ S_k(\nu')(x_{k}) - \rho(x_{k}) \Big]- m \Big[ S_m(\nu')(c') - \rho(c') \Big]\geq 0 \\
  	\nonumber \Rightarrow \qquad&  - m \Big[ S_m(\nu')(c') - \rho(c') \Big]   \geq 0  \\
  	\Rightarrow \qquad& S_m(\nu')(c') \leq \rho(c').
  	\end{align}

   Now consider the spectral projections $p_l:= \chi_{(\frac{1}{l}, \infty )}(c)$ of $c$ for all $l \in \N$. Now for $y \in \CM$ with $0 \leq y \leq 1$ write $y_l= p_l y p_l$ and observe that 
   \begin{align*}
       y_l \leq p_l \leq l c.
   \end{align*}

   Now since Eq. \ref{a} is true for all $m \in [n]$ and $0 \leq c' \leq c$, in particular, if we take $c'= \frac{y_l}{l}$, we obtain $S_m(\nu')(y_l) \leq \rho(y_l)$ which implies
   \begin{align}\label{b}
       S_m(\nu')(p_l y p_l) \leq \rho(p_l y p_l).
   \end{align}

   Furthermore, $ p_l \xrightarrow{k \ri \infty } e_{n, \eps} $ in $\sigma$-weak topology. Therefore, taking limit as $l \to \infty$ in Eq. \ref{b} we obtain, for all $m \in [n]$
   \begin{align}\label{b'}
       S_m(\nu')(e_{n, \eps} y e_{n, \eps}) \leq \rho(e_{n, \eps} y e_{n, \eps}).
   \end{align}

   Now, for $m=1$ and $y= \sum_{k=1}^{n} x_i$, we obtain 
   \begin{align}\label{c}
       \nu'(e_{n, \eps} \sum_{k=1}^{n} x_i) \leq \rho( e_{n, \eps} \sum_{k=1}^{n} x_i), ~ \text{since, }~  e_{n, \eps} (\sum_{k=1}^{n} x_i) e_{n, \eps} = e_{n, \eps} \sum_{k=1}^{n} x_i.
   \end{align}

   On the other hand, consider  the point $(T(x_2), \ldots, T(x_n), 0) \in R$. Since $T(1) \leq 1$, we note that  $(T(x_2), \ldots, T(x_n), 0) \in \mathcal{L}$.
  	Hence,  we have  
  	\begin{align*}
  	\kappa(\underbar{x}) \geq \kappa(T(x_2), \ldots, T(x_n), 0 ).
  	\end{align*}
  	From this we deduce the following
  	\begin{align*}
  	& \kappa(\underbar{x}) \geq \kappa( T(x_2), \ldots, T(x_n), 0 ) \\
  	\Rightarrow 	\qquad&  
  	\sum_{k=1}^{n}  k\Big[ S_k(\nu')(x_{k}) - \rho(x_{k}) \Big]\\
     &\geq \\
  	\qquad 	&\sum_{k=1}^{n}  k\Big[ S_k(\nu')(T(x_{k+1})) - \rho(T(x_{k+1})) \Big]\\
  	\Rightarrow 	\qquad&  
  	\sum_{k=1}^{n}  k\Big[ S_k(\nu')(x_{k}) - \rho(x_{k}) \Big]\\
  	&\geq \\
  	\qquad 	&\sum_{k=1}^{n}  k\Big[ S_k(T^*(\nu')(x_{k+1})) - \rho(x_{k+1}) \Big]\\
  	\Rightarrow \qquad 	&  \nu'( \sum_{k=1}^n x_k) \ge  \rho ( \sum_{ k =1}^n x_k).
  	\end{align*}

   Now, combined with Eq. \ref{c} we obtain
   \begin{align*}
       \nu'((1- e_{n, \eps}) \sum_{ k =1}^n x_k ) \geq \rho( (1- e_{n, \eps}) \sum_{ k =1}^n x_k ),
   \end{align*}
   which implies 
   \begin{align*}
       \norm{\nu'} \geq \rho( (1- e_{n, \eps})), ~ \text{ since } (1- e_{n, \eps}) \sum_{ k =1}^n x_k = 1- e_{n, \eps}.
   \end{align*}
   This completes the proof.
\end{proof}

	We note that to obtain our result, we must know the maximal inequality of an associated sequence of Martingales. In \cite{cuculescu1971martingales}, the authors obtain weak type maximal inequality for a sequence of $L^1$-bounded Martingales for a tracial von Neumann algebra. However, the techniques used in the proof of \cite[Proposition 5]{cuculescu1971martingales} or Theorem \ref{maximal lemma-1} are not useful to obtain similar inequalities for certain Martingales in the state preserving dynamical system. To overcome such difficulties we make use of the following result of Neveu (\cite{neveu1964deux}) [also see \cite[Lemma 2]{dang1979pointwise}].
	
	\begin{thm}\label{dang-njoc lemma}
		Let $\CM$ be a von Neumann algebra with a f.n. positive linear functional $\rho$. Also, suppose $ (\CM_n)_{n \in \N}$ be a decreasing sequence of von Neumann subalgebras of $ \CM$ equipped with a sequence of f.n. conditional expectations
		\begin{align*}
			\CE_n : \CM \to \CM_n \text{ with } \rho\circ \CE_n = \rho \text{ for all } n \in \N.
		\end{align*}
		If $ 0= a_1 < a_2 < \cdots <a_n <\cdots < 1 $ and $\lim a_n = 1$, then the operator 
		\begin{align*}
			\CT = \sum_{ n =1}^\infty( a_{n+1} -a_n )\CE_n 
		\end{align*}
		satisfies $\CT(1)=1$ and $\rho(\CT(a))= \rho(a)$ for all $a \in \CM$. Further,  for any given $\eps >0$,  there exists such a sequence $(a_n)$ and a sequence of integers $ q_1 < q_2 \cdots $, such that 
		\begin{align*}
			\sum_{ n=1 }^{\infty } \norm{ \frac{1}{q_n}\sum_{ k=0 }^{q_n}    \CT^k -\CE_n } \leq \eps .
		\end{align*}
	\end{thm}

	A metric measure space $(X,d, \mu)$ is a metric space $(X, d)$ equipped with a radon measure $\mu$ satisfying $0< \mu(V(s,r))< \infty$ for all $s \in X$ and $r>0$, where $V(s,r):= \{ t \in X: d(s,t)\leq r\}$. The measure $\mu$ is said to satisfy the \textit{doubling condition} if there exists an universal constant $c_1>0$ such that 
	\begin{align}\label{doub cond-defn}
		\mu(V(s, 2r)) \leq c_1 \mu(V(s,r)) ~\text{for all } s \in X \text{ and } r>0.
	\end{align}

	We now recall the following result from \cite[Corllary 7.4]{hytonen2012systems} (also cf. \cite[Lemma 4.2]{Hong2021}), which acts as one of the main tools to obtain our maximal inequality.
	
	\begin{lem}\label{Hytonen partition}
		Let $(X,d, \mu)$ be a metric measure space satisfying the doubling condition. Then there exists a finite collection of families $\CP^1,\ldots, \CP^\kappa$ , where $ \kappa \in \N$ and  each $\CP^i:= (\CP^i_n) $ is a sequence of partitions of $X$ such that the following are true:
		\begin{enumerate}
			\item for each $ i \in [\kappa]$ and  $n \in \Z$, the partition $\CP^i_{n}$ is a refinement of $\CP^i_{n+1}$;
			\item there is a constant $c_2>0$ such that, for all $s \in X$ and $r>0$, there exists $ i \in [\kappa]$, $n \in \Z$, and $Q \in \CP^i_n$ such that 
			\begin{align*}
				V(s,r) \subset Q,  ~ \text{and }~ \mu(Q) \leq c_2 \mu(V(s,r)).
			\end{align*}
		\end{enumerate}
	\end{lem}
 
\noindent The following remark will be used in the sequel.

 \begin{rem}\label{Hytonen rem}
     It follows from \cite{hytonen2012systems} that in Lemma \ref{Hytonen partition}, suppose  $i \in [\kappa]$ and  $n \in \Z$, then for each element $Q \in \CP^i_n$, there exists two balls $B_1 $ and $B_2$ of finite radius such that
     $$ B_1 \subseteq Q \subseteq B_2.$$
  Moreover, we note that  when $X$ is a locally compact group and $\mu$ is a Haar measure, we can conclude that $\mu(Q) < \infty$ for all such $Q$.
 \end{rem}

	\noindent Consider the families $\CP^1,\ldots, \CP^\kappa$ as in Lemma \ref{Hytonen partition}. For $s \in X$ and $n \in \Z$, let $\CP^i_n(s)$ denote the unique cube containing $s$, in the family $\CP^i$ of partitions. Suppose $ F $ is a compact subset of $X$. For any  given $r \in \R_+$,  we consider the following set 
	\begin{align*}
		O_r(F) =\underset{s \in F}{\cup} \{ n_r(s)\in \Z :&
		\text{ such that } V(s, r)  \subseteq \CP^i_{ n_r(s)}(s), \text{ for some }   i \in  [\kappa] 
		,  ~ \text{ and }~ \\
		&\textcolor{black}{ \mu(\CP^i_{ n_r(s)}(s)) \leq c_1c_2 \mu(V(s,r))} \}.
	\end{align*}
	
	\begin{lem}\label{finite set}
		Let $(X,d, \mu)$ be a metric measure space satisfying the doubling condition. Furthermore, let $r>0$ and assume that $\mu(V(s,r))= \mu(V(t,r))$ for all $s,t \in X$. Now suppose  $ F \subseteq X$ is a compact set. Then there exists a finite subset $O'_r(F)$ of $O_r(F)$ such that  
		\begin{align*}
			\forall ~ s \in F ~ \exists~ n_r(s) \in O'_r(F) &\text{ such that } V(s, r)  \subseteq \CP^i_{ n_r(s)}(s), \text{ for some }   i \in  [\kappa], \text{ and } \\
			&\mu(\CP^i_{ n_r(s)}(s)) \leq c_1 c_2 \mu(V(s,r)).
		\end{align*} 
	\end{lem}

	\begin{proof}
		Let $ r >0$. Consider the open cover $ \{ V(s, r ) : s \in F \}$ of $F$. Thus, as $F$ is compact,  so, there exists a finite subcover $ V(s_1, r ), V(s_2, r ), \cdots, V(s_l, r )$ such that 
		\begin{align*}
			F \subseteq \overset{l}{ \underset{m=1}{\cup}} V(s_{ m}, r ).
		\end{align*}
		Now, by Lemma \ref{Hytonen partition}, we have for all $m \in [l] $ there exists $ n_{2r}(s_m) \in \Z$  and  a  $  i \in  [\kappa]$ such that 
		\begin{align*}
			V(s_m, 2r ) \subseteq \CP_{n_{2r}(s_m) }^i (s_m) \text{ and } \mu( \CP_{n_{2r}(s_m) }^i (s_m)  ) \leq c_2 \mu (   V(s_m, 2r ) ) .
		\end{align*}
		
		\noindent For all $m \in [l]$, choose only one such $n_{2r}(s_m) \in \Z$ and denote the collection of these integers to be $O'_r(F)$. Therefore, the set $O'_r(F)$ is finite. Moreover, for any $ s \in F$,  there exists $  m \in  [l]$, $i \in [\kappa]$ and $n_{2r}(s_m) \in O'_r(F)$ such that
		\begin{align*}
			V(s,r) \subseteq V(s_m,2r) \subseteq \CP_{n_{2r}(s_m)}^i(s_m).
		\end{align*}
		Also, from Eq. \ref{doub cond-defn} we have,
		\begin{align*}
			\mu( \CP_{n_{2r}(s_m) }^i (s_m)  ) 
   \leq c_2 \mu (   V(s_m, 2r ) )
   \leq c_1c_2 \mu (   V(s_m, r ) )= c_1 c_2 \mu (   V(s, r ) ).
		\end{align*}
		This completes the proof.
	\end{proof}
	

Let $(X,d, \mu)$ be a metric measure space satisfying the doubling condition and $\CB$ denotes the $\sigma$-algebra of Borel sets in $X$.  Suppose $\{\CP^i_n : n \in \Z \text{ and } i \in [\kappa] \}$ be  the collection of partitions of $X$ as obtained in Lemma \ref{Hytonen partition}. We note that  for $ i \in [\kappa] \text{ and }  n \in \Z$, $\CP^i_{n}$ is a refinement of $\CP^i_{n+1}$ and let $\CB^i_n$ denote the subalgebra of $\CB$ generate by $\CP_n$. Therefore, for each $ i \in [ \kappa ], ~\{\CB^i_n\}_{n \in \Z}$ is a decreasing sequence of sub $\sigma$-algebras of $\CB$.

Let $\CM$ be a von Neumann algebra with a f.n. state $\rho$  on $\CM$. We consider the von Neumann algebra $\CN:= L^\infty(X, \CB, \mu) \otimes \CM$. Clearly, $\CN$ is equipped with the f.n. weight $\tilde{\rho} := (\int \cdot~ d\mu) \otimes \rho$. Now for all $n \in \Z$ we define
	\begin{align*}
		\CN_n:= L^\infty(X, \CB_n, \mu) \otimes \CM.
	\end{align*}
	Note that, $(\CN_n)_n$ is a decreasing sequence of subalgebras of $\CN$ and therefore, for all $n \in \Z$, by \cite{takesaki1972conditional} there exists a normal conditional expectation $E_n$ from $\CN$ onto $\CN_n$ since $ \tilde{\rho}|_{ \CN_n}$ is semifinite. Observe that, the conditional expectation $E_n$ is of the form $E_n:= \E_n(\cdot | \CB_n) \otimes I$, where $\E_n(\cdot | \CB_n)$ is the conditional expectation from $L^\infty(X, \CB, \mu)$ onto $L^\infty(X, \CB_n, \mu)$. Observe that, $\tilde{\rho} \circ E_n= \tilde{\rho}$ for all $n \in \N$. Moreover, from \cite[Theorem 1.22.13]{sakai1971c} and \cite[Theorem 2.1.6]{dunford1940linear} we can also identify
	
	\begin{align*}
		\CN_* =  
  L^1(X;  \CM_*).
\end{align*}

\begin{rem}
 If $T: \CN \ri \CN$ is a bounded linear map. Suppose $ f \in \CN_*= L^1(X, \CM_*)$. 
 Then we use the following identification in the sequel:
 \begin{align*}
 	T(f)(x) = \int f(s)(( T(x)(s))d\mu.
 \end{align*}
\end{rem}

\noindent Now we have the following theorem.
\begin{thm}\label{dominated by martingale}
		For $r>0$ let $A_r$ be the averaging operator
		\begin{align*}
			A_r f(s):= \frac{1}{\mu(V(s,r))} \int_{V(s,r)} f d\mu; ~ f \in L^1(X; M_*).
		\end{align*}
		Then there exists a $ \kappa \in \N$, and martingales $ \{( E^i_k)_{k \in \Z}: i \in [\kappa]\}  $ on $\CN$ such that for any given $r > 0$ and   $ s \in X$, there exists $ i \in [\kappa] $ and a $k_r \in \Z $ satisfying
		\begin{align*}
			A_r(f)(s)	\leq c_2 E_{k_r}^i(f)(s).
		\end{align*}
	\end{thm}
	
	\begin{proof}
		Applying Lemma \ref{Hytonen partition} we obtain a $ \kappa \in \N$,  a constant $c_2 $ and martingales $ \{( E^i_k)_{k \in \Z}: i \in [\kappa]\}  $ on $\CN$. Observe that for all $f \in L^1(X; \CM_*)$ and $s \in X$,
		\begin{align*}
			E^i_k(f)(s):= \frac{1}{\mu(\CP^i_k(s))} \int_{\CP^i_k(s)} f d\mu.
		\end{align*}
		
		Furthermore, for all $s \in X$ and $r>0$, there exists $ i \in [\kappa]$ and $k_r \in \Z$ such that 
		\begin{align*}
			V(s,r) \subset \CP^i_{k_r}(s),  ~ \text{and }~ \mu(\CP^i_{k_r}(s)) \leq c_2\mu(V(s,r)).
		\end{align*}
		
		Therefore,
		\begin{align*}
			A_r f(s):= \frac{1}{\mu(V(s,r))} \int_{V(s,r)} f d\mu 
			& \leq c_2 \frac{1}{\mu(\CP^i_{k_r}(s))} \int_{\CP^i_{k_r}(s)} f d\mu\\
			& =c_2E^i_{k_r}f(s)
		\end{align*}
		%
		%
		%
		%
		%
		%
		%
		%
	\end{proof}

\noindent  As a corollary of Theorem \ref{dominated by martingale} and Lemma \ref{finite set}, we obtain the following.

 \begin{cor}\label{dominated by martingale-cor}
     Assume that $G$ is a group metric measure space satisfying doubling condition and let $F \subseteq G$ be compact and $r>0$. Then for all $s \in F$ there exists $ i \in [\kappa] $ and $k_r \in  O'_{r}(F) $ satisfying
		\begin{align*}
			A_r(f)(s)	\leq C E_{k_r}^i(f)(s),
		\end{align*}
  where $C=c_1c_2$
 \end{cor}

 \begin{proof}
     Let $r>0$. By Lemma \ref{finite set}, we have for all $s \in F$ there exists $ i \in [\kappa]$ and $k_r \in  O'_{r}(F)$ such that 
		\begin{align*}
			V(s,r) \subset \CP^i_{k_r}(s),  ~ \text{and }~ \mu(\CP^i_{k_r}(s)) \leq c_1c_2\mu(V(s,r)).
		\end{align*}
		Therefore,
		\begin{align*}
			A_r f(s):= \frac{1}{\mu(V(s,r))} \int_{V(s,r)} f d\mu 
			& \leq c_1c_2 \frac{1}{\mu(\CP^i_{k_r}(s))} \int_{\CP^i_{k_r}(s)} f d\mu\\
			& = c_1c_2E^i_{k_r}f(s),
		\end{align*}
  which completes the proof.
 \end{proof}

	Let $G$ be an amenable locally compact group equipped with a right invariant Haar measure $\mu$. A metric $d$ on $G$ is said to be invariant if 
	\begin{align}\label{inv metric-defn}
		d(e,g)=d(h, gh), \text{ for all } g,h \in G,
	\end{align}
		where $e$ is the identity element of $G$. Furthermore, we also assume that $(G, d, \mu)$ is a metric measure space satisfying the doubling condition. That is there exists a constant $c_1>0$ (cf. Eq. \ref{doub cond-defn}) such that 
	\begin{align*}
		\mu(V_{2r}) \leq c_1 \mu(V_r) \text{ for all } r>0,
	\end{align*}
	where for any $r>0$, $V_r:=\{s \in G: d(s,e) \leq r\}$. Note that it follows from \cite{calderon1953general} that $G$ is unimodular. We will call the triple $(G,d, \mu)$ a group metric measure space.

	Now let us assume a dynamical system $(\CM, G, \alpha)$, where $(G,d, \mu)$ be a group metric measure space satisfying the doubling condition. Then we define for all $r>0$,
	\begin{align*}
		A_r(x)= \frac{1}{\mu(V_r)} \int_{V_r} \alpha_s(x) d\mu(s), ~ x \in \CM,
	\end{align*}
	
	\noindent and for all $\phi \in \CM_*$, we define for $r>0$,
	\begin{align*}
		A_r(\phi)(x)= \frac{1}{\mu(V_r)} \int_{V_r} \phi(\alpha_s(x)) d\mu(s), ~ x \in \CM.
	\end{align*}
	
\noindent	Finally, for $f \in L^1(G; \CM_*)$ we consider the following averages;
	\begin{align*}
		A_r'f(s)= \frac{1}{\mu(V_r)} \int_{V_r} f(ts) d\mu(t), ~ s \in G, ~r>0.
	\end{align*}
	
	Note that by invariance of $d$, for all $s \in G$ and $r>0$, we have $V_r s= V(s,r)$ and $\mu(V_r)=\mu(V_r s)$. Therefore, for all $r>0$,
	\begin{align*}
		A_r'f(s)= \frac{1}{\mu (V(s,r))} \int_{V(s,r)} f d\mu.
	\end{align*}

	Let $(r_n)_{n \in \N}$ be any increasing sequence of strictly positive real numbers. We will now work with the averages $(A_{r_n})$ and $(A'_{r_n})$. But for notational convenience, we will denote them as $(A_n)$ and $(A'_n)$. Moreover, for every $n \in \N$, by $\mu_n$ we denote the normalization of  $\mu$ restricted  on $V_{r_n}$, i.e,  $\mu_n(E)  = \frac{ \mu( V_{r_n} \cap E ) }{ \mu(V_{r_n})}$.  Then we have the following,
	\begin{align*}
		A_n(\phi)(x)= \int_G \phi(\alpha_s(x)) d\mu_n(s), ~ \phi \in \CM_*,~ x \in \CM, ~ n \in \N
	\end{align*}
 and 
	\begin{align*}
		A_n'f(s)= \int_{G} f(ts) d\mu_n(t), ~f \in L^1(G; \CM_*),~ s \in G, ~ n \in \N.
	\end{align*}
	
	\begin{defn}\label{weak11}
		Let $\phi \in M_{* +}$. Then a sequence of positive maps $( S_n)$ on $M_*$ 
  is said to have weak $(1,1)$ property if  for all $n \in \N$ and $\eps, \delta>0$, there exists projections $\{e^{ \delta}_{n, \eps}   ~\}$ such that
		\begin{enumerate}
			
\item $ \rho( 1- e^\delta_{ n, \eps})< \frac{\norm{\phi}}{\eps}$ and 
   
			\item $S_k(\phi)(e^{ \delta}_{n, \eps} x e^{ \delta}_{n, \eps}) \leq \delta \norm{\phi} \norm{x} + \eps \rho(e^{\delta}_{n, \eps} x e^{ \delta}_{n, \eps})$ for all $x \in M_+$ and $k \in \{1,\cdots, n\}$.
		\end{enumerate}
	\end{defn}
	
	The following theorem proves a weak type $(1,1)$-property for Martingales.
	
	\begin{thm}\label{maximal ineq for martingale}
		Let $\CM$ is a von Neumann algebra with f.n. state $ \rho$ and $ ( E_n)_{ n \in \Z}$ be a decreasing sequence of $\rho$-preserving martingales. Then for  $\phi \in \CM_{* +}$ and $l \in \N$,  the sequence $(E_{r-l-1}(\phi))_{r \geq 1}$ satisfies the following  properties: for all $n \in \N$ and $\eps, \delta>0$, there exists a projection $e^{\delta}_{n, \eps} \in \CM$ such that  
		
		\begin{enumerate}
			\item $\rho(1- e^{\delta}_{n, \eps}) < \frac{\norm{\phi}}{\eps}$,
\item $E_{k-l-1}(\phi)( x ) \leq L_k^{\delta }( \phi)(x)  + \eps \rho(x)$ for all $x \in e^{\delta}_{n, \eps} \CM_+ e^{\delta}_{n, \eps}$ and $k \in [n]$ and 
\item $E_{k-l-1}(\phi)( x ) \leq \delta \norm{\phi} \norm{x} + \eps \rho(x)$ for all $x \in e^{\delta}_{n, \eps} \CM_+ e^{\delta}_{n, \eps}$ and $k \in [n]$.
		\end{enumerate}
	\end{thm}
	
	\begin{proof}
		Let $n \in \N$ and $\eps, \delta>0$. Since $ ( E_{k-l-1})_{ k \geq 1}$ is a decreasing  sequence of martingales, then by Theorem \ref{dang-njoc lemma} there exists $T_\delta$ and a sequence of positive integers $(q^\delta_k)_{k \geq 1}$  such that 
		\begin{align*}
			\sum_{ n=1 }^{\infty } \norm{ \frac{1}{q^\delta_n}\sum_{ k=0 }^{q^\delta_n}    T_{\delta}^k -E_{n-l-1} } \leq \delta.
		\end{align*}
		Moreover, we have $ \rho \circ T_\delta = \rho $. Suppose $ S^\delta_n ( \cdot)= \displaystyle \frac{1}{n}\sum_{ k=0 }^{n-1}    T_{\delta}^k( \cdot)$ and  $ L_n^{\delta}( \cdot) = E_{n-l-1}( \cdot)- S^\delta_{q_n^\delta} ( \cdot)$. Then it follows that 
		\begin{align*}
			E_{n-l-1}( \cdot)=  L_n^{\delta }( \cdot) + S^\delta_{q_n^\delta} ( \cdot).
		\end{align*}
		For $ q_n^\delta $ and $\eps >0 $, we employ  Theorem \ref{maximal lemma-1}, to find    $e^\delta_{n, \eps} \in \CP(\CM)$ such that 
		\begin{enumerate}
			\item 	$\rho(1-e^\delta_{n, \eps}) < \frac{\norm{\phi}}{\eps }$ and 
			\item $ 	S^\delta_k(\phi)(x) \leq \eps \rho(x) \text{ for all } x \in e^\delta_{n, \eps} \CM_+ e^\delta_{n, \eps} \text{ and  }  k \in [ q_n^\delta].$
		\end{enumerate}

		Now suppose $x \in e^\delta_{n, \eps} \CM_+ e^\delta_{n, \eps}$ and $k \in [n]$, then we have,
		\begin{align*}
			 E_{k-l-1}(\phi)(x) \leq&  L_k^{\delta }( \phi)(x)   + S^\delta_{q_k^\delta} ( \phi)(x) \\
			\leq &     \delta \norm{\phi} \norm{x}  + \eps\rho(x).
		\end{align*}
		This completes the proof.
	\end{proof}

\noindent At this point, we would like to remark on the following for future reference.

\begin{rem}\label{max ieq-rem}
In our context we consider  $\CN = L^\infty(X, \CB, \mu) \otimes \CM$,  $\CN_n= L^\infty(X, \CB_n, \mu) \otimes \CM$ and 
    a sequence of conditional expectation  $(E_n)$, where $E_n$ is a map from $\CN$ onto $\CN_n$ as discussed before. Here, we note that   $E_n= \E_n(\cdot | \CB_n) \otimes I$, where $\E_n(\cdot | \CB_n)$ is the conditional expectation from $L^\infty(X, \CB, \mu)$ onto $L^\infty(X, \CB_n, \mu)$. Observe that, $\tilde{\rho} \circ E_n= \tilde{\rho}$ for all $n \in \N$. Note that for each $l > 0$, $( \E_{ k-l-1})_{ k \in \N}$ is decreasing sequence of martingales. So, by Theorem \ref{dang-njoc lemma},  for given $\delta >0$, there exists $ \T_\delta$ (which is  sum of $\E_n$) such that 
   \begin{align*}
			\sum_{ n=1 }^{\infty } \norm{ \frac{1}{q_n}\sum_{ k=0 }^{q_n}    \T^k_\delta -\E_{n-l-1} } \leq \delta .\\
		\end{align*}
Now suppose $ \CT_\delta = \T_\delta \otimes 1$, then we have 
\begin{align*}
			\sum_{ n=1 }^{\infty } \norm{ \frac{1}{q_n}\sum_{ k=0 }^{q_n}    \CT^k_\delta -E_{n-l-1} } \leq \delta .
		\end{align*}

Further, suppose $ \L_n^\delta = \frac{1}{q_n}\sum_{ k=0 }^{q_n}    \T^k_\delta -\E_{n-l-1}$  and $ L_n^\delta = \frac{1}{q_n}\sum_{ k=0 }^{q_n}    \CT^k_\delta -E_{n-l-1}$. Then observe that $ L_n^\delta = \L_n^\delta \otimes 1$. As $ \int \E_n(f) g d\mu = \int  f \E_n(g) d\mu, \text{ for all } f, g \in L^\infty( X, \CB, \mu ) \text{ and } n \in \Z$, it follows that 
$$  \int_X \L_n^\delta (f) g d\mu = \int f \L_n^\delta (g) d \mu \text{ for all } f, g \in L^\infty( X, \CB, \mu ) \text{ and } n \in \N.$$
\end{rem}

\noindent With the notations as in Remark \ref{max ieq-rem}, we have the following lemma.

\begin{lem}\label{maximal ineq for averages-lem}
    Let $f(t):= \chi_D(t) \alpha_t^*(\phi)$ and $x= \chi_F \otimes y$, where $D$ is a compact subset of $X$, $F$ be a measurable subset and $y \in \CM$. Then 
    \begin{align}
        \abs{L^{ \delta}_k (f)(x)} \leq \delta \mu(F) \norm{y} \norm{\phi}.
    \end{align}
\end{lem}

\begin{proof}
   It follows from Remark \ref{max ieq-rem}. Indeed, we observe that 
    \begin{align*}
        \abs{L^{ \delta}_k (f)(x)}
        &=\abs{f(L^{ \delta}_k (x))}\\
        &=\abs{\int_X f(s)( L^{ \delta}_k (x)(s)) d \mu(s) }\\ 
        &= \abs{\int_X \chi_D(s) \L^{\delta}_k(\chi_F)(s) \alpha_s^*(\phi)(y) d \mu(s)}\\
        & \leq \norm{\phi} \norm{y} \abs{ \int_X \chi_D(s) \L^{\delta}_k(\chi_F)(s) d \mu(s) }\\
        & \leq \norm{\phi} \norm{y} \int_F \L^{\delta}_k (\chi_D)(s) d \mu(s)\\
        &\leq \norm{\phi} \norm{y} \norm{\L^{ \delta}_k (\chi_D)} \mu(F) \\
        &\leq \norm{\phi} \norm{y} \norm{\L^{ \delta}_k} \mu(F) 
        \leq \delta \norm{\phi} \norm{y} \mu(F).
    \end{align*}
    This completes the proof.
\end{proof}

	For the next theorem, we assume that $\CM$ is a von Neumann algebra with f.n.  state $\rho$ and  $(G,d, \mu)$ be a group metric measure space satisfying doubling condition. Furthermore, $(\CM,G, \alpha)$ is a non-commutative dynamical system such that  $\rho \circ \alpha_s = \rho$ for all $s \in G$.
Now using Theorem \ref{maximal ineq for martingale}, we will prove a weak type $(1,1)$- property for ergodic averages.

\begin{thm}\label{maximal ineq for averages}
		Let $\phi \in \CM_{* +}$. Then the sequence $(A_k(\phi))_{k \geq 1}$ has the property: there exists $\kappa \in \N$ such that for all $n \in \N$ and $\eps, \delta, \upsilon>0$, there exists projections $\{e^{i, \delta, \ups}_{n, \eps}:  i \in [\kappa] ~\}$ such that for all $ i \in [\kappa]$
		\begin{enumerate}
			\item $\rho(1- e^{i, \delta, \ups}_{n, \eps}) < \frac{1+ \ups }{\eps}   \norm{  \phi} +\ups $
			\item $A_k(\phi)(x) \leq C\delta \norm{\phi} \norm{x} +  C\eps \rho(x)$ for all $x \in e^{i, \delta, \ups}_{n, \eps} \CM_+ e^{i, \delta, \ups}_{n, \eps}$ and $k \in [n]$.
		\end{enumerate}
		The constant $C=c_1c_2$ only depends on the 	group $G$.	
	\end{thm}
	
	\begin{proof}
		Let $n \in \N$ and $\eps, \delta, \ups>0$. Further let $K:= \cup_{l=1}^n supp(\mu_l)$. Then $K$ is compact. By amenability of $G$, choose a compact set $F$ such that 
		\begin{align*}
			\frac{m(KF)}{m(F)} < 1+ \ups.
		\end{align*}
		 As $ e \in K$, we note that  $ F \subseteq KF$. Now define a function 
		\begin{align*}
			f(t):= \chi_{KF}(t) \alpha^*_t(\phi).
		\end{align*}
		
		\noindent Observe that $f \in \CN_*$ and we have the following
		\begin{align*}
			\alpha^*_s A_n(\phi)= A_n'f(s) \text{  for all }  s \in F .
		\end{align*}
		
		Furthermore, by Lemma \ref{finite set}, we notice that  for all $ k \in [n]$, the set $O'_{r_k}(F) \subseteq \Z$ is finite. Then find a $ l \in \N$ such that 
		\begin{align*}
			\overset{n}{ \underset{k=1}{\cup}} O'_{r_k}(F)  \subseteq [-l \mathrel{{.}\,{.}}\nobreak l]
		\end{align*}

Now we fix a $n \in \N$. Suppose $m_n = 2\text{ max}\{ l, n\} $.  Fix $i\in [\kappa] $  and let $ t \in \{ -l, \cdots, m_n\} $. Since $\{ \CP_t^i\}$ is a decreasing sequence of partition, so, by using Remark \ref{Hytonen rem} and the compactness of $KF$,   we can find $ \{\CP^i_t( j_p^t) \}_{ p=1}^{s_t}$
such that 
\begin{align*}
    KF \subseteq \cup_{ p =1}^{ s_t} \CP^i_t( j_p^t) \subseteq \cup_{ p =1}^{ {s_{m_n}   }}  \CP^i_{m_n}( j_p^{m_n}),
\end{align*}
where $\CP^i_t( j_p^t)$ are finitely many elements of the partition $\CP^i_t$. Suppose 
$$ Y^i_{t} = \cup_{ p =1}^{ {s_{m_n}   }}  \CP^i_{t}( j_p^{t}),~ t \in \{ -l, \cdots, m_n \}.$$
Then observe that 
\begin{align*}
    Y^i_{-l} \subseteq Y^i_{-(l-1)}\subseteq \cdots \subseteq  Y^i_{m_n}.
\end{align*}

Now consider $M_{t}^i = L^\infty( Y_{m_n}^i,  \CB^i_t|_{ Y^i_{m_n}}, \mu)$  for all $ t \in \{ -l, \cdots, m_n\} $, where $ \CB^i_t$ be the $\sigma$-algebra generated by $\CP^i_t$.
Now we have the following ;
\begin{enumerate}
    \item $\mu( Y^i_{m_n}) < \infty $ (cf. Remark \ref{Hytonen rem}),

    \item $M_{m_n}^i \otimes \CM \subseteq M_{m_n-1}^i \otimes \CM \subseteq \cdots \subseteq M_{-l}^i \otimes \CM \subseteq L^\infty( Y^i_{ m_n}, \CB, \mu)\otimes \CM $ and

    \item $ E_t^i( L^\infty( Y^i_{ m_n}, \CB, \mu)\otimes \CM ) \subseteq L^\infty( Y^i_{ m_n}, \CB, \mu)\otimes \CM  $ for all $ t \in \{ -l, \cdots, m_n\} $, as $E_t^i(g)(s)= \frac{1}{\mu(\CP^i_t(s))} \int_{\CP^i_t(s)} g(z) d \mu(z)$ for all $g \in L^\infty( Y^i_{ m_n}, \CB, \mu)\otimes \CM$.\\
\end{enumerate}

We use the same notation for  the restriction of $E_t^i $ on   $L^\infty( Y^i_{ m_n}, \CB, \mu)\otimes \CM $ also use the same notation for the restriction of $ \tilde{\rho}$ on $ L^\infty( Y^i_{ m_n}, \CB, \mu)\otimes \CM $, which is a f.n positive linear functional as $\mu(Y_{m_n}^i) < \infty$.

Thus,  we note that  for each  $ i \in  [\kappa]$,  $ ( E_k^i)_{ k = -l}^{ m_n} $ is a \textcolor{black}{decreasing} finite  sequence of martingales on $ L^\infty( Y^i_{ m_n}, \CB, \mu)\otimes \CM $ preserving  the f.n. positive linear functional $ \tilde{\rho}$.  Therefore,  
by repeated use of Theorem \ref{maximal ineq for martingale}, we obtain projections $(e ^{1,\delta}_{n, \eps}), \ldots, (e^{\kappa, \delta}_{n, \eps})$ in $L^\infty( Y^i_{ m_n}, \CB, \mu)\otimes \CM \subseteq \CN$ such that for all $ i \in [\kappa]$
		\begin{enumerate}
			\item $\tilde{\rho}(1-e^{i, \delta}_{n, \eps}) < \frac{\norm{f}_1}{\kappa \eps } $  and
			

            \item[(2)] $E_{k-l-1}^{i}(f)(x) \leq  L_k^{i, \delta}(f)(x) + \eps\widetilde{\rho}( x  ) $ for all $ x \in e_{n, \eps}^{i, \delta} \CN_+ e_{n, \eps}^{i, \delta}$ and $ k \in [2\max \{n,l\}]$.
		\end{enumerate}

    Suppose, $x=1_{E} \otimes y$, with  $y \in \CM_+$, and $E$ is measurable subset of $Y^i_{m_n}$. Then we note the following;

    \begin{enumerate}
        \item $E_{k-l-1}^{i}(f)(1_E \otimes y)= \int_E E_{k-l-1}^{i}(f)(s) y  d \mu(s)$,

        \item $\abs{L_k^{i, \delta}(f)(1_E \otimes y)} \leq \delta \mu(E) \norm{\phi} \norm{y}$, (cf. Lemma \ref{maximal ineq for averages-lem}),

        \item $\tilde{\rho}(1_E \otimes y)= \int_E \rho(y) d \mu(s)$.
    \end{enumerate}

\noindent    Therefore, varying $x$ over the characteristic function of  the form $1_{E} \otimes y$, with  $y \in \CM_+$, we obtain

    \begin{align}\label{eq. 1}
        E_{k-l-1}^i(f)(s)(e_{n, \eps}^{i, \delta}(s) y e_{n, \eps}^{i, \delta}(s ) ) 
        &\leq  \delta \norm{\phi} \norm{y} +  \eps \rho(  e_{n, \eps}^{i, \delta}(s) y e_{n, \eps}^{i, \delta}(s)  )
		\end{align}
		for all  $y \in \CM_+$  and a.e $s \in F$. Also by Corollary \ref{dominated by martingale-cor}, we note that for all $k \in [n]$ and $s \in F$ there exists $k_r \in	\overset{n}{ \underset{k=1}{\cup}}  O'_{r_k}(F)$ such that for all $y \in \CM_+$
		\begin{align*}
			A_k'(f)(s)( e_{n, \eps}^{i_s, \delta}(s)  \alpha_{  s^{-1} }(y) e_{n, \eps}^{i_s, \delta}(s) )
			\leq 
			C E_{k_r}^{i_s}(f)(s) ( e_{n, \eps}^{i_s, \delta}(s)  \alpha_{  s^{-1} }(y ) e_{n, \eps}^{i_s, \delta}(s) ).
		\end{align*}
		
		\noindent	Then, for a.e. $s \in F$ denoting $p_s:= e_{n, \eps}^{i_s, \delta}(s)$ we have for all
		$y \in M_+$ and $ k \in [n]$, 
		
		\begin{align*}
			A_k(\phi) (  \alpha_s( p_s)  y  \alpha_s( p_s)      ) =& \alpha^*_s A_k(\phi)( p_s  \alpha_{  s^{-1} }(y ) p_s )\\
			=& A_k'(f)(s)( p_s  \alpha_{  s^{-1} }(y ) p_s )\\
			\leq & 
			C E_{k_r}^{i_s}(f)(s) ( p_s  \alpha_{  s^{-1} }(y ) p_s )\\
			\leq& C \delta \norm{\phi} \norm{y} + C \eps \rho( p_s  \alpha_{  s^{-1} }(y ) p_s ), \text{ as } k_r \in  [-l \mathrel{{.}\,{.}}\nobreak l]  \\
			\leq& C \delta \norm{\phi} \norm{y} + C \eps \rho(  \alpha_s( p_s)  y  \alpha_s( p_s)).
		\end{align*} 
		
		\noindent 	Therefore, for a.e. $s \in F$, and $y \in M_+$ and $ k \in [n]$, we have 
		\begin{align*}
			A_k(\phi) (  \alpha_s( e_{n, \eps}^{i_s, \delta}(s))  y  \alpha_s( e_{n, \eps}^{i_s, \delta}(s))      )
			\leq C \delta \norm{\phi} \norm{y} + C \eps \rho(  \alpha_s( e_{n, \eps}^{i_s, \delta}(s))  y  \alpha_s( e_{n, \eps}^{i_s, \delta}(s)).
		\end{align*}
		
\noindent		Finally note that for all $i \in [\kappa] $, there exists $s_{i, \ups } \in F$ such that 
		\begin{align*}
			\rho(1- \alpha_{s_{i, \ups}}( e^{  i_{ s_{i, \ups}, \delta}}_{n, \eps}(s_{i, \ups})))= \rho(1-  e^{  i_{ s_{i, \ups}, \delta}}_{n, \eps}(s_{i, \ups})) \leq \inf_{s \in F} \rho(1- e^{i_s, \delta}_{n, \eps}(s)) + \ups,
		\end{align*}
		Now suppose  $q_{n, \eps}^{i,\delta, \ups} := \alpha_{s_{i, \ups}}( e^{  i_{ s_{i, \ups} , \delta}}_{n, \eps}(s_{i, \ups}))$.
		Then we note that 
		\begin{align*}
			\rho(1- q_{n, \eps}^{i,\delta, \ups}) &\leq \frac{1}{\mu(F)} \sum_{ i=1}^\kappa\int_F \rho(1- e^{i, \delta}_{n ,\eps}(s)) d\mu(s) + \ups\\
			&\leq \frac{1}{\mu(F)} \sum_{ i=1}^\kappa \int_G \rho(1- e^{i, \delta}_{n, \eps}(s)) d\mu(s) + \ups \\
			& \leq \frac{1}{\mu(F)}    \frac{ \kappa  \norm{   f}}{ \kappa \eps} + \ups \\
			& \leq \frac{\mu( KF)}{\mu(F)}    \frac{   \norm{  \phi}}{ \eps} + \ups \\
			& \leq \frac{1+ \ups }{\eps}   \norm{  \phi} +\ups.
		\end{align*} 
		
		\noindent	Furthermore,   for all $ x \in \CM_+ $, $ i \in [\kappa]$ and $ k \in [n] $, we observe that
		\begin{align*}
			A_k(\phi) (  q_{n, \eps}^{i,\delta, \ups}  x  q_{n, \eps}^{i,\delta, \ups} ) 	=&  A_k(\phi) ( \alpha_{s_{i, \ups}}( e^{  i_{ s_{i, \ups} , \delta}}_{n, \eps}(s_{i, \ups}))  x \alpha_{s_{i, \ups}}( e^{  i_{ s_{i, \ups} , \delta}}_{n, \eps}(s_{i, \ups}))   ) \\
			\leq& C \delta \norm{\phi} \norm{x} + C \eps \rho(  \alpha_{s_{i, \ups}}( e^{  i_{ s_{i, \ups} , \delta}}_{n, \eps}(s_{i, \ups}))  x  \alpha_{s_{i, \ups}}( e^{  i_{ s_{i, \ups} , \delta}}_{n, \eps}(s_{i, \ups} )))\\
			\leq & C \delta \norm{\phi} \norm{x} + C \eps \rho(  q_{n, \eps}^{i,\delta, \ups}  x  q_{n, \eps}^{i,\delta, \ups} ).
		\end{align*}
		This completes the proof.  
	\end{proof}

\begin{cor}\label{maximal ineq for averages-cor}
		Let $\phi \in \CM_{* s}$. Then the sequence $(A_k(\phi))_{k \geq 1}$ has the property: there exists $\kappa \in \N$ such that for all $n \in \N$ and $\eps, \delta, \upsilon>0$, there exists projections $\{e^{i, \delta, \ups}_{n, \eps}:  i \in [\kappa] ~\}$ such that for all $ i \in [\kappa]$
		\begin{enumerate}
			\item $\rho(1- e^{i, \delta, \ups}_{n, \eps}) < \frac{1+ \ups }{\eps}   \norm{  \phi} +\ups $
			\item $\abs{A_k(\phi)(x)} \leq C\delta \norm{\phi} \norm{x} +  C\eps \rho(x)$ for all $x \in e^{i, \delta, \ups}_{n, \eps} \CM_+ e^{i, \delta, \ups}_{n, \eps}$ and $k \in [n]$.
		\end{enumerate}
  
		The constant $C=c_1c_2$ only depends on the 	group $G$.	
	\end{cor}

 \begin{proof}
     Let $\phi \in \CM_{* s}$. Note that, there exists $\phi_1, \phi_2 \in \CM_{* +}$ such that $\phi= \phi_1 - \phi_2$ and $\norm{\phi}= \norm{\phi_1}+ \norm{\phi_2}$. Then by applying Theorem \ref{maximal ineq for averages} for $\psi= \phi_1 + \phi_2$ we obtain for all $n \in \N$ and $\eps, \delta, \upsilon>0$, there exists projections $\{e^{i, \delta, \ups}_{n, \eps}:  i \in [\kappa] ~\}$ such that for all $ i \in [\kappa]$

		\begin{enumerate}
			\item $\rho(1- e^{i, \delta, \ups}_{n, \eps}) < \frac{1+ \ups }{\eps}   \norm{  \psi} +\ups \leq \frac{1+ \ups }{\eps}   \norm{  \phi} +\ups$, since $\norm{\psi} \leq \norm{\phi}$,
			\item $A_k(\psi)(x) \leq C\delta \norm{\psi} \norm{x} +  C\eps \rho(x)$ for all $x \in e^{i, \delta, \ups}_{n, \eps} \CM_+ e^{i, \delta, \ups}_{n, \eps}$ and $k \in [n]$,
		\end{enumerate}

  where $C=c_1c_2$. Furthermore, for all $x \in e^{i, \delta, \ups}_{n, \eps} \CM_+ e^{i, \delta, \ups}_{n, \eps}$ we have
  \begin{align*}
      \abs{A_k(\phi)(x)}
      &=\abs{A_k(\phi_1 - \phi_2)(x)}\\
      &\leq \abs{A_k(\phi_1 + \phi_2)(x)}
      = A_k(\psi)(x) \\
      &\leq C\delta \norm{\psi} \norm{x} +  C\eps \rho(x)\\
      &\leq C\delta \norm{\phi} \norm{x} +  C\eps \rho(x), \text{ since } \norm{\psi} \leq \norm{\phi}.
  \end{align*}
  This completes the proof.
 \end{proof}

	\begin{rem}
		We note that given  $\varphi \in \CM_*$ and  the sequence $(A_k(\varphi))_{k \geq 1}$ satisfies the conclusions of the Theorem  \ref{maximal ineq for averages}. Then actually, one can call that  $(A_k(\varphi))_{k \geq 1}$ satisfies weak type $(1, 1)$ if it satisfies the conclusion of the Theorem  \ref{maximal ineq for averages}. But, for the notational simplicity we will not do that. We point out that the difference between the standard definition of weak type $(1,1)$ with our definition  of weak type $(1, 1)$ (see Definition \ref{weak11}) are more dependency on the collection of projections and extra error term ($\delta \norm{\phi}\norm{x}$ in \ref{weak11}). 
	\end{rem}

	\section{\textbf{Pointwise ergodic theorems} }\label{Pointwise ergodic theorems}
	
	In this section, we will study pointwise convergence of specific ergodic averages. With the weak type maximal inequality in our disposal (cf. \S \ref{Weak type inequality}), now we need  to prove an appropriate Banach principle and pointwise convergence of elements from a dense subset. 
	
	Throughout this section, $\CM$ is  assumed to be a von Neumann algebra with a f.n. tracial state $\tau$ and $(G,d, \mu)$ is a group metric measure space satisfying doubling condition. We further assume that $(\CM, G, \alpha)$ is a non-commutative dynamical system such that there exists a f.n. state $\rho$ on $\CM$ such that $\rho \circ \alpha_s = \rho$ for all $s \in G$.
	
	Before we proceed to the main results of this section, let us recall the following structural property of $(G,d, \mu)$. For proof of this fact, we refer to \cite[Proposition 17]{tessera2007volume} (also cf. \cite[Lemma 6.9]{Hong2021}). 
	
	\begin{lem}\label{lacunary seq}
		There exists a lacunary sequence $(r_i)_{i \in \N}$ in $\R_+$ such that for all $i \in \N$  
		\begin{enumerate}
			\item[(i)] $2^i \leq r_i < 2^{i+1}$ and 
			\item[(ii)] $\frac{\mu(V(x, r_i) \setminus V(x, r_i- (3/2)^i))}{\mu(V(x, r_i))} \leq C' (3/4)^i$ for all $x \in G$,
		\end{enumerate}
		where the constant $C'$ depends only on the doubling condition.
	\end{lem}
	
	As we have $\rho \circ \alpha_s = \rho$ for all $s \in G$,  observe that one can define a contraction $u_s$ on the $GNS$ space $L^2(\CM, \rho)$ defined by 
	\begin{align*}
		u_s(x \Omega_\rho)= \alpha_s(x) \Omega_\rho,
	\end{align*}
	satisfying $u_s \circ u_t= u_{st}$ for all $s,t \in G$, where $\Omega_\rho$ is the associated cyclic, separating vector in $L^2(\CM, \rho)$. Note that the Hilbert space inner product and norm in the aforesaid $GNS$ space will be denoted by $\inner{\cdot}{\cdot}_\rho$ and $\norm{\cdot}_\rho$ respectively. Now we have the following mean ergodic theorem.
	
	\begin{thm}\label{Mean erg}
		Let $(r_i)_{i \in \N}$ be the lacunary sequence as in Lemma \ref{lacunary seq}. Then the averages 
		\begin{align*}
			A_{r_i}x := \frac{1}{\mu(V_{r_i})} \int_{V_{r_i}} u_s(x) d \mu(s), ~ x \in L^2(\CM, \rho)
		\end{align*}
		converges in $\norm{\cdot}_\rho$.
	\end{thm}
	
	\begin{proof}
		Since $(L^2(\CM, \rho), G, u)$ forms a non-commutative dynamical system , we observe that 
		\begin{align*}
			L^2(\CM, \rho):= F_2 \oplus F_2^\perp,
		\end{align*}
		where 
		\begin{align*}
			F_2:= \{x \in L^2(\CM, \rho): u_s(x)=x~ \forall s \in G\} \\
			F_2^\perp:= \overline{span} \{ y- u_s y : y \in L^2(\CM, \rho),~ s \in G \}.
		\end{align*}
		
		Therefore, it is enough to check the convergence of $A_{r_i}x$ for $x= y- u_s y$ where $y \in L^2(\CM, \rho),~ s \in G$. Moreover, note that
		\begin{align*}
			A_{r_i}x= A_{r_i}^1 y - A_{r_i}^2 y,
		\end{align*}
		where
		\begin{align*}
			A_{r_i}^1 y:= \frac{1}{\mu(V_{r_i})} \int_{V_{r_i} \setminus (V_{r_i} \cap V_{r_i} s)} u_s(y) d \mu(s),\\
			A_{r_i}^2 y:= \frac{1}{\mu(V_{r_i})} \int_{V_{r_i}s \setminus (V_{r_i} \cap V_{r_i} s)} u_s(y) d \mu(s).
		\end{align*}
		
		Now since $(3/2)^i \rightarrow \infty$ as $i \rightarrow \infty$, we can choose $i \in \N$ such that $(3/2)^i \geq d(e,s)= \abs{s}$. Moreover, $V_{r_i} \setminus (V_{r_i} \cap V_{r_i} s) \subseteq V_{r_i} \setminus V_{r_i - (3/2)^i}$ and $V_{r_i}s \setminus (V_{r_i} \cap V_{r_i} s) \subseteq V_{r_i}s \setminus V_{r_i - (3/2)^i} s$. Therefore,
		\begin{align*}
			\norm{A_{r_i}^1 y} \leq \frac{\mu(V_{r_i} \setminus V_{r_i - (3/2)^i})}{\mu(V_{r_i})} \norm{y} \leq C (3/4)^i \norm{y}, ~ \text{and } \\
			\norm{A_{r_i}^2 y} \leq \frac{\mu( V_{r_i}s \setminus V_{r_i - (3/2)^i} s )}{\mu(V_{r_i} s)} \norm{y} \leq C (3/4)^i \norm{y}.
		\end{align*}
		
		Therefore, $\norm{A_{r_i} x}$ converges to $0$ as $i \to \infty$.
	\end{proof}
	
	We recall the following Lemma from \cite{Bik-Dip-neveu}.
	
	\begin{lem}\cite[Lemma 4.5]{Bik-Dip-neveu} \label{key lem1}
		Let $(\CM,G, \alpha)$ be a non-commutative dynamical system with f.n. state $\rho$ as described in the beginning of the section. Then we have a non-commutative dynamical system $(\CM',G, \alpha')$, where $\CM'$ is the commutant of $\CM$ such that 
		\begin{align*}
			\inner{\alpha'_{s}(y')x \Omega_\rho}{\Omega_\rho}_\rho= \inner{y' \alpha_{s^{-1}}(x) \Omega_\rho}{\Omega_\rho}_\rho
		\end{align*}
		for all $x \in \CM$, $y' \in \CM'$ and $s \in G$.
	\end{lem}
	
	For the group $G$, observe that $\beta= \{\beta_s\}_{s \in G}:= \{\alpha^*_{s^{-1}}\}$ defines an action of the group $G$ on $\CM_*$. We consider the following averages.
	\begin{align*}
		B_{r_i}(\nu) := \frac{1}{\mu(V_{r_i})} \int_{V_{r_i}} \beta_s(\nu) d \mu(s),~ \nu \in \CM_*,
	\end{align*}
	where $(r_i)_{i \in \N}$ is the Lacunary sequence obtained in Lemma \ref{lacunary seq}.

Although the next result is present in \cite{Bik-Dip-neveu}, we include it for the sake of completeness.
 
	\begin{thm}\label{mean erg thm for state}
		For all $\nu \in \CM_*$ there exists $\bar{\nu} \in \CM_*$ such that 
		\begin{align*}
			\bar{\nu}= \norm{\cdot}_1- \lim_{i \to \infty} B_{r_i}(\nu).
		\end{align*}
	\end{thm}
	
	\begin{proof}
		For all $s \in G$, consider the contraction $u_s$ on $L^2(\CM, \rho)$ by
		\begin{align*}
			u_s(x \Omega_\rho)= \alpha_s (x) \Omega_\rho, ~ x \in \CM
		\end{align*}
		and for all $i \in \N$, define
		\begin{align*}
			T_i (\cdot):= \frac{1}{\mu(V_{r_i})} \int_{V_{r_i}} u^*_{s^{-1}}(\cdot) d \mu(s). 
		\end{align*}
		Then by Theorem \ref{Mean erg}, $T_i(\xi)$ converges in $L^2(\CM, \rho)$ as $i \to \infty$ for all $\xi \in L^2(\CM, \rho)$.
		
		Furthermore, observe that for all $y_1, y_2 \in \CM'(\sigma^\rho)$ there exists $z \in \CM$ such that $y_1^*y_2 \Omega_\rho= z \Omega_\rho$, where $\CM'(\sigma^\rho)$ is the strong* dense subspace of analytic elements in $\CM'$ (cf. \cite{Falcone:2000th}). Let us define $\psi_{y_1, y_2} \in \CM_*$ by
		\begin{align*}
			\psi_{y_1, y_2}(x)= \inner{x y_1 \Omega_\rho}{y_2 \Omega_\rho}_\rho, ~ x \in \CM.
		\end{align*}
		Finally for all $x \in \CM$,
		\begin{align*}
			B_{r_i}(\psi_{y_1, y_2}) (x) 
			&:= \frac{1}{\mu(V_{r_i})} \int_{V_{r_i}} \inner{\alpha_{  s^{-1} }(x) y_1 \Omega_\rho}{y_2 \Omega_\rho}_\rho d \mu(s)\\
			&= \frac{1}{\mu(V_{r_i})} \int_{V_{r_i}} \inner{\alpha_{  s^{-1} }(x)  \Omega_\rho}{z \Omega_\rho}_\rho d \mu(s)~ (\text{by Lemma \ref{key lem1}})\\
			&= \frac{1}{\mu(V_{r_i})} \int_{V_{r_i}} \inner{x  \Omega_\rho}{u^*_{s^{-1}}(z \Omega_\rho)}_\rho d \mu(s)\\
			&= \inner{x \Omega_\rho}{T_i(z \Omega_\rho)}_\rho.
		\end{align*}
		
		Now let $\eta = \lim_{i \to \infty} T_i(z \Omega_\rho) \in L^2(\CM, \rho)$. Then clearly, for all $x \in \CM$, $B_{r_i}(\psi_{y_1, y_2}) (x) \rightarrow \inner{x \Omega_\rho}{\eta}_\rho$ as $i \to \infty$. Define,
		\begin{align*}
			\overline{\psi_{y_1, y_2}} (x):= \inner{x \Omega_\rho}{\eta}_\rho, ~ x \in \CM.
		\end{align*}
		
		Clearly, $\overline{\psi_{y_1, y_2}} \in \CM_*$ and $\overline{\psi_{y_1, y_2}} \circ \alpha_s= \overline{\psi_{y_1, y_2}}$ for all $s \in G$. Now the result follows from the fact that the set $\{\psi_{y_1, y_2} : y_1, y_2 \in \CM'(\sigma^\rho) \}$ is total in $\CM_*$.
	\end{proof}
	
	\begin{lem}\label{dense set-lem}
		Let $(E, G, \gamma)$ be a Dynamical system, where $(E, \norm{\cdot})$ is a Banach space. Then for every $k \in \N$ and $a \in E$
		\begin{align*}
			\lim_{n \to \infty} \norm{B_n (a - B_k(a))} = 0,
		\end{align*}
		where, 
		\begin{align*}
			B_m(a):= \frac{1}{\mu(V_{r_m})} \int_{V_{r_m}} \gamma_s(a) d \mu(s),~ m \in \N.
		\end{align*}
	\end{lem}
	
	\begin{proof}
		For all  $n \in \N$, note that
		\begin{align*}
			B_n(a - B_k(a)) 
			&= \frac{1}{\mu(V_{r_n})} \int_{V_{r_n}} \Big[ \gamma_s(a) - \gamma_s (B'_k(a)) \Big] d\mu(s) \\
			&= \frac{1}{\mu(V_{r_n})} \int_{V_{r_n}} \Big[ \gamma_s (a) - \gamma_s \Big( \frac{1}{\mu(V_{r_k})} \int_{V_{r_k}} \gamma_t (a) d\mu(t) \Big) \Big] d\mu(s) \\
			&= \frac{1}{\mu(V_{r_n})} \frac{1}{\mu(V_{r_k})} \int_{V_{r_n}} \int_{V_{r_k}}  \Big( \gamma_s (a)- \gamma_{st}(a) \Big) d\mu(t) d\mu(s) \\
			&= \frac{1}{\mu(V_{r_n})} \frac{1}{\mu(V_{r_k})} \int_{V_{r_k}} \int_{V_{r_n}}  \Big( \gamma_s (a)- \gamma_{st}(a) \Big) d\mu(s) d\mu(t) \quad \text{ (using Fubini)} \\
			&= \frac{1}{\mu(V_{r_n})} \frac{1}{\mu(V_{r_k})} \int_{V_{r_k}} \Big(\int_{V_{r_n}}   \gamma_s (a)d\mu(s)- \int_{V_{r_n}} \gamma_{st}(a)d\mu(s) \Big) d\mu(t)   \\
			&= \frac{1}{\mu(V_{r_k})} \int_{V_{r_k}} \frac{1}{\mu(V_{r_n})} \Big(\int_{V_{r_n}}   \gamma_s (a)d\mu(s)- \int_{V_{r_n}t} \gamma_s (a)d\mu(s) \Big) d\mu(t).
		\end{align*}
		
		Now,
		\begin{align*}
			&\int_{V_{r_n}}   \gamma_s (a)d\mu(s)- \int_{V_{r_n}t} \gamma_s (a)d\mu(s) \\
			&= \int_{V_{r_n} \setminus V_{r_n} \cap V_{r_n} t}   \gamma_s (a)d\mu(s)- \int_{V_{r_n}t \setminus V_{r_n} \cap V_{r_n} t } \gamma_s (a)d\mu(s).
		\end{align*}
		Now since $(3/2)^n \rightarrow \infty$ as $n \rightarrow \infty$, we can choose $n \in \N$ such that $(3/2)^n \geq d(e,t)= \abs{t}$. Moreover, $V_{r_n} \setminus (V_{r_n} \cap V_{r_n} t) \subseteq V_{r_n} \setminus V_{r_n - (3/2)^n}$ and $V_{r_n}t \setminus (V_{r_n} \cap V_{r_n} t) \subseteq V_{r_n}t \setminus V_{r_n - (3/2)^n} t$. Therefore,
		
		\begin{align*}
			\norm{\frac{1}{\mu(V_{r_n})} \int_{V_{r_n} \setminus V_{r_n} \cap V_{r_n} t}   \gamma_s (a)d\mu(s) } 
			\leq \frac{\mu(V_{r_n} \setminus V_{r_n - (3/2)^n})}{\mu(V_{r_n})} \norm{a} \leq C' (3/4)^n \norm{a}, ~ \text{and } \\
			\norm{\frac{1}{\mu(V_{r_n})} \int_{V_{r_n}t \setminus V_{r_n} \cap V_{r_n} t } \gamma_s (a)d\mu(s) } 
			\leq \frac{\mu( V_{r_n}t \setminus V_{r_n - (3/2)^n}t )}{\mu(V_{r_n} t)} \norm{a} \leq C' (3/4)^n \norm{a}.
		\end{align*}
		Finally, we obtain
		\begin{align*}
			\norm{B_n(a - B_k(a))} \leq 2C' (3/4)^n \norm{a},
		\end{align*}
		which proves the result.
	\end{proof}
	
	For the non-commutative dynamical system $(\CM, G, \alpha)$ with f.n. state $\rho$ satisfying $\rho \circ \alpha_s = \rho$ for all $s \in G$, by Theorem \ref{mean erg thm for state}, it follows that $\norm{\cdot}_1- \lim_{i \to \infty} B_{r_i}(\nu)$ exists for all $\nu \in \CM_*$ and the limit is denoted by $\bar{\nu}$. Now we have the following proposition which identifies a dense subset of $\CM_*$ such that for the elements of the set, the ergodic averages converges in certain sense which is close to the bilateral almost uniform convergence.
	
	\begin{prop}\label{dense set}
		Consider the following set 
		\begin{align*}
			\CW_1 := \{ \nu - B_{r_k}(\nu ) + \bar{\nu} : \quad  k \in \N,\  \nu \in \CM_{*+} \text{ with } \nu \leq \lambda \rho \text{ for some } \lambda >0 \}.
		\end{align*}
		\begin{enumerate} 
			\item[(i)] Write $\CW= \CW_1-\CW_1$, then $\CW$ is dense in $\CM_{* s}$ and 
			
			\item[(ii)] for all $\nu \in \CW$, we have  
			\begin{align}\label{maximal type}
				\lim_{n \to \infty} \sup_{x \in \CM_+, x \neq 0} \abs{(B_{r_n}(\nu)- \bar{\nu})(x)}/ \rho(x) =0.
			\end{align}
		\end{enumerate}
	\end{prop}
	
	\begin{proof}
		\emph{ (i):} Proof follows from the similar argument as in the proof of \cite[Lemma 4.9, Part (i)]{Bik-Dip-neveu}.\\
		
		\emph{(ii):}  Fix $k \in \N$ and consider $ \nu_k := \nu - B_{r_k}(\nu ) + \bar{\nu} $ and it is enough to prove eq. \ref{maximal type} for $\nu_k $. Now since 
		\begin{align*}
			B_{r_n}(\nu_k) - \bar{\nu}= B_{r_n}(\nu - B_{r_k}(\nu))
		\end{align*}
		for all $n \in \N$, we note from Lemma \ref{dense set-lem} it follows that $\overline{\nu_k}= \bar{\nu}$. 
		Since $\nu \leq \lambda \rho $, so by \cite[Lemma 4.4]{Bik-Dip-neveu} there exists a unique $y_1' \in M_+'$ with $  y_1' \leq \lambda $ such that 
		\begin{align*}
			\nu(x)= \langle y_1'x \Omega_{\rho}, \Omega_{\rho} \rangle_{\rho} \text{ for all } x \in \CM.
		\end{align*}
		
		Let  $y' \in M'$, write 
		$$ B_{r_n}'(y') := \frac{1}{\mu(V_{r_n})}\int_{V_{r_n}} \alpha_{s}'(y') d \mu(s)$$ and by Lemma \ref{key lem1}, we have
		\begin{align*}
			\langle B_{r_n}'(y') x \Omega_{\rho}, \Omega_{\rho} \rangle_{\rho} = \langle y' B_{r_n}(x) \Omega_{\rho}, \Omega_{\rho} \rangle_{\rho}, \quad x \in \CM,  y' \in \CM',
		\end{align*}
		where 
		$$B_{r_n}(x):= \frac{1}{\mu(V_{r_n})}\int_{V_{r_n}} \alpha_{s^{-1}}(x) d \mu(g).$$ Now  for all $n \in \N$ and $x \in \CM_+$, we note that 
		\begin{align*}
			\abs{(B_{r_n}(\nu_k) - \bar{\nu})(x)}
			&= \abs{(\nu_k - \bar{\nu})(B_{r_n}(x))}\\
			&= \abs{(\nu - B_{r_k}(\nu))(B_{r_n}(x))} ~~ \text{ (since } \nu_k := \nu - B_{r_k}(\nu ) + \bar{\nu})\\
			&= \abs{\nu(B_{r_n}(x)) - \nu(B_{r_k}(B_{r_n}(x)))} ~~ \text{ (as } B_{r_k}(\nu)(\cdot) = \nu(B_{r_k}(\cdot)))\\
			&= \abs{\inner{y_1' B_{r_n}(x) \Omega_{\rho} }{\Omega_{\rho}}_{\rho} - \inner{y_1' B_{r_k}(B_{r_n}(x))\Omega_{\rho}}{\Omega_{\rho}}_{\rho}} ~~ \text{ (as } \nu(\cdot) = \langle   y_1' (\cdot)\Omega_{\rho}, \Omega_{\rho} \rangle_{\rho}) \\
			&= \abs{\inner{y_1' B_{r_n}(x) \Omega_{\rho} }{\Omega_{\rho}}_{\rho} - \inner{B_{r_k}'(y_1') (B_{r_n}(x))\Omega_{\rho}}{\Omega_{\rho}}_{\rho}}, \text{ (by  Lemma } \ref{key lem1}) \\
			&= \abs{\inner{(y_1'- B_{r_k}'(y_1'))B_{r_n}(x) \Omega_{\rho}}{\Omega_{\rho}}_{\rho}} \\
			&= \abs{\inner{B_{r_n}'(y_1'- B_{r_k}'(y_1'))x \Omega_{\rho}}{\Omega_{\rho}}_{\rho}}, \text{ (by  Lemma } \ref{key lem1}) \\
			&\leq \norm{B_{r_n}'(y_1'- B_{r_k}'(y_1')} \rho(x).\\
		\end{align*}
		
\noindent		Moreover, $\lim_{n \to \infty} \norm{B_{r_n}'(y_1'- B_{r_k}'(y_1')}=0$ follows from Lemma \ref{dense set-lem}. Finally, we obtain
		\begin{align*}
			\lim_{n \to \infty} \sup_{x \in M_+, x \neq 0} \abs{(B_{r_n}(\nu_k)- \bar{\nu}_k)(x)}/ \rho(x)
			&=\lim_{n \to \infty} \sup_{x \in M_+, x \neq 0} \abs{(B_{r_n}(\nu_k)- \bar{\nu})(x)}/ \rho(x) \\
			&=0.
		\end{align*}
		This completes the proof.		
	\end{proof}
	
	\begin{rem}\label{dense set-rem1}
		\begin{enumerate}
			\item We first remark that for all $p \in \CP(\CM)$ and $\nu \in \CW$ it follows that
			\begin{align*}
				\lim_{n \to \infty} \sup_{x \in p\CM_+p, x \neq 0} \abs{(B_{r_n}(\nu)- \bar{\nu})(x)}/ \rho(x) =0.
			\end{align*}

			\item Let $\CM$ be a von Neumann algebra equipped with a f.n. tracial state $\tau$ and $\nu \in \CW$. Since $\rho \in \CM_{* +}$, there exists a unique $X\in L^1(\CM,\tau)_+$  such that $\rho(x)= \tau(Xx)$ for all $x \in 	\CM$. Then for any $s>0$ consider the 	projection $q_{s}:= 	\chi_{(1/s, s)}(X) \in \CM$.  Observe that $\tau(1-q_{s})\xrightarrow{s \rightarrow \infty} 0$. Therefore, for every $\epsilon>0$ there exists a $s_0>0$ such that $\tau(1-q_{s_0})< \epsilon$. Further, it implies $X q_{s_0} \leq {s_0} q_{s_0}$.  Thus,  for all $0\neq x \in q_{s_0}\CM_+ 	q_{s_0}$ we have 
			\begin{align*}
				\begin{split}
					\frac{\rho(x)}{\tau(x)} = 	\frac{\tau(Xx)}{\tau(x)}&= 
					\frac{\tau(X q_{s_0} x)}{\tau(x)} \quad  \text{ (since $q_{s_0} x= x$)}\\
					&\leq  	\frac{\tau(s_0 x)}{\tau(x)}  = {s_0}.
				\end{split} 
			\end{align*}

			Therefore, it follows from Proposition \ref{dense set} that
			\begin{align*}
				\lim_{n \to \infty} \sup_{x \in p\CM_+p, x \neq 0} \frac{\abs{(B_{r_n}(\nu)- \bar{\nu})(x)}}{\tau(x)} \leq s_0 \lim_{n \to \infty} \sup_{x \in p\CM_+p, x \neq 0} \frac{\abs{(B_{r_n}(\nu)- \bar{\nu})(x)}}{\rho(x)}=0.
			\end{align*}
			\item Let the group $G$ under consideration is unimodular. Then since the disks $V_{r_n}$ are symmetric set, we note that 
			\begin{align}\label{symm avg}
				A_{r_n}(\nu):= \frac{1}{\mu(V_{r_n})} \int_{V_{r_n}} \alpha^*_s(\nu) d \mu(s)= \frac{1}{\mu(V_{r_n})} \int_{V_{r_n}} \alpha^*_{s^{-1}}(\nu) d \mu(s) =: B_{r_n}(\nu)~ \forall ~ \nu \in \CM_*.
			\end{align}
		\end{enumerate}
	\end{rem}

	For the next results, we assume that our von Neumann algebra $\CM$ is equipped with a f.n. tracial state $\tau$. 
	
	\begin{defn}\label{kernel defn}
		Let $(\CM, G, \alpha)$ is a non-commutative dynamical system (cf. Definition  \ref{nc dyn sys}) such that $\rho \circ \alpha_s = \rho$ for all $s \in G$ for some f.n. state $\rho$ on $\CM$, where, $G$ is a amenable group metric measure space satisfying the doubling condition. We call the quadruple $(\CM, G, \alpha, \rho)$ a kernel. 
	\end{defn}

	Our next result proves an appropriate Banach principle which is important to obtain an ergodic convergence. Before we begin, first we define bilateral almost uniform (b.a.u.) convergence for a sequence $(\nu_n)_{n \in \N}$ in $\CM_*$.
	
	\begin{defn}\label{bau conv- functionals}
		Let $(\nu_n)_{n \in \N}$ be a sequence in $\CM_*$ and $\nu \in \CM_*$. We say that $\nu_n$ converges in bilateral almost uniform (b.a.u.) topology if for every $\epsilon>0$ there exists $p \in \CP(\CM)$ with $\tau(1-p)<\epsilon$ such that
		\begin{align*}
			\lim_{n \to \infty} \sup_{x \in p\CM_+p, x \neq 0} \abs{(\nu_n- \nu)(x)}/ \tau(x) =0.
		\end{align*}
	\end{defn}
	
	\begin{rem}\label{equiv of bau convg}
		\begin{enumerate}
			\item Let $(\nu_n)_{n \in \N}$ be a sequence in $\CM_*$ and $\nu \in \CM_*$. Then by duality of $M_*$ and $L^1(\CM, \tau)$ (cf. Theorem \ref{predual prop}) we obtain a sequence $(X_n)_{n \in \N}$ and $X$ in $L^1(\CM, \tau)$. We observe that if $(\nu_n)_{n \in \N}$ converges to $\nu$ in b.a.u. in the sense of Definition \ref{bau conv- functionals}, then the associated sequence $(X_n)$ will converge to $X$ in b.a.u. in the sense of Definition \ref{bau conv defn}. Indeed, if $(\nu)_n$ converges to $\nu$ in b.a.u., then for all $\epsilon>0$ there exists $e \in \CP(\CM)$ with $\tau(1-e)< \epsilon$ such that for all $\delta>0$ there exists $n_0 \in \N$ such that $n \geq n_0$ implies
			\begin{align*}
				&\abs{(\nu_n - \nu)(exe)} \leq \delta \tau(exe) ~ \forall x \in \CM_+ \setminus \{0\}\\
				\Leftrightarrow & \abs{ \tau((X_n - X)(exe))} \leq \delta \tau(exe) ~ \forall x \in \CM_+ \setminus \{0\}\\
				\Leftrightarrow & \abs{ \tau(e(X_n - X)ex)} \leq \delta \tau(exe) ~ \forall x \in \CM_+ \setminus \{0\}\\
				\Leftrightarrow & \norm{e(X_n - X)e} \leq \delta.
			\end{align*}
			Therefore, $X_n$ converges to $X$ in b.a.u.

			\item On the other hand, if the sequence $(X_n)_{n \in \N}$ and $X$ in $L^1(\CM, \tau)$ such that $X_n$ converges to $X$ in b.a.u., then the argument above shows that the sequence $(\nu_n)_{n \in \N}$ in $\CM_*$ associated to $(X_n)_{n \in \N}$ converges b.a.u. to the element $\nu \in \CM_*$ associated to $X$ in the sense of Definition \ref{bau conv- functionals}.
		\end{enumerate}
	\end{rem}
	
	\begin{prop}\label{complete wrt bau-functionals}
		Let $(\nu_n)_{n \in \N}$ be a sequence in $\CM_*$ which is Cauchy in b.a.u. in the sense of Definition \ref{bau conv- functionals}. Then there exist $\nu \in \CM_*$ such that $\nu_n$ converges to $\nu$ in b.a.u.
	\end{prop}
	
	\begin{proof}
		The proof follows from Theorem \ref{complete wrt bau} and Remark \ref{equiv of bau convg}.
	\end{proof}
	
	\begin{thm}\label{Banach principle}
		Let $(\CM, G, \alpha, \rho)$ be a kernel and $\tau$ be a f.n. tracial state on $\CM$. Suppose  $\phi \in \CM_*$ and  the sequence $(A_k(\phi))_{k \geq 1}$ satisfies the conclusions of the Theorem  \ref{maximal ineq for averages}. Then the following set 
		$$ \CC =  \{   \phi \in \CM_{*}   : A_n(\phi) \text{ converges in b.a.u } \} $$
		is a closed set in $ \CM_{*}$. 
	\end{thm}

	\begin{proof}
		Let  $(\phi_n) $ be a sequence in $ \CC$,  converging to $ \phi \in \CM_*$ in norm, i.e, 
		$$ \limn \norm{ \phi_n -\phi} = 0.$$ 
		Then, we wish to show that $ \phi \in  \CC$. 
		First we note that, since $\rho \in \CM_{* +}$, there exists a unique $X\in L^1(\CM,\tau)_+$  such that $\rho(x)= \tau(Xx)$ for all $x \in 	\CM$. Then for any $s>0$ consider the 	projection $q_{s}:= 	\chi_{(1/s, s)}(X) \in \CM$.  Observe that $(1-q_{s})\xrightarrow{s \rightarrow \infty} 0$ in SOT. 
		
		Let $\upsilon>0$ be arbitrary. Therefore, there exists a $s_0>0$ such that $\tau(1-q_{s_0})< \frac{\upsilon}{2}$. Further, it implies $X q_{s_0} \leq {s_0} q_{s_0}$.  Thus,  for all $0\neq x \in q_{s_0}\CM_+ 	q_{s_0}$ we have \\
		\begin{align}\label{rho-tau}
			\begin{split}
				\frac{\rho(x)}{\tau(x)} = 	\frac{\tau(Xx)}{\tau(x)}&= 
				\frac{\tau(X q_{s_0} x)}{\tau(x)} \quad  \text{ (since $q_{s_0} x= x$)}\\
				&\leq  	\frac{\tau(s_0 x)}{\tau(x)}  = {s_0}.
			\end{split} 
		\end{align}

		Since $ \limn \norm{ \phi -\phi_n} = 0 $, so,  find  a sequence $\{\phi_{n_1}, \phi_{n_2}, \cdots \} $ satisfying 
		\begin{align*}\label{coneq-3}
			(1) ~&n_1 <n_2< \cdots \\
			(2) ~& \norm{\phi- \phi_{n_j}} < \frac{1}{4^j} \text{ for  all }  j \in \N .
		\end{align*}

		If possible, by replacing $\phi_{n_j}$ with  $\phi_j$, for the simplicity of  notation,    we can assume 
		
		$$	\norm{\phi- \phi_{j}} < \frac{1}{4^j} \text{ for  all }  j \in \N.$$
		
		As $ \phi_j \in \CC$, so there exists a $ f_j \in \CP(\CM) $ such that 
		\begin{align*}
			\begin{split}
				(1)~&\tau(1-f_j)  < \frac{1}{2^j}  \text{ for  all } j\in \N, \\
				(2) ~	&\sup_{x \in f_jM_+f_j, x \neq 0} \frac{\abs{\Big( A_n(\phi_j) - A_m(\phi_{j})\Big)(x)}}{\tau(x)} < \frac{1}{2^{j}} \text{ for all }  n,m \geq  N_j  \\
				(3)~& N_1 < N_2 < \cdots.
			\end{split}
		\end{align*}

		Further, note that  $ \phi - \phi_{j} \in \CM_{* s}$ for all $j \in \N$.  Now using Corollary \ref{maximal ineq for averages-cor} for $\eps=\ups=\frac{1}{2^j},~ \delta= \frac{1}{2^k}$, and $N_j \in \N$,  we obtain  projections $\{(q_j^{i,k})_{j,k \in \N}:  i \in [\kappa]\}$ such that for all $ i \in [\kappa]$
		\begin{enumerate}
			\item $\rho(1- q_j^{i,k}) \leq \frac{1+ \ups }{\eps}   \norm{ \phi - \phi_{j}} + \ups\\
			\leq \frac{1+ \frac{1}{2^j} }{2^j} + \frac{1}{2^j} = \frac{1}{2^{j-1}} + \frac{1}{2^{2j}}$  for all $j,k \in \N$, and \\
			\item $A_n(\phi - \phi_{j}) (  q_j^{i,k}  x  q_j^{i,k} ) \leq \frac{\norm{\phi - \phi_{j} }\norm{x}}{2^k} + \eps \rho(   q_j^{i,k} x    q_j^{i,k})\\
			= \frac{1}{2^{k+ 2j}}\norm{x} + \frac{1}{2^j} \rho(   q_j^{i,k} x    q_j^{i,k})$ for all $ n \in \{ 1, 2, \cdots, N_{j+1} \}$ and for all $k \in \N$.
		\end{enumerate}
		
		Now it immediately  follows that for all $ i \in [\kappa]$ and for all $k \in \N$,  $\rho(1-q_j^{i,k})$ converges to $0$ as $j$ tends to $\infty$. We note that this convergence is uniformly over in the variable $k$. Therefore, since $\tau$ is a f.n. state, so, it implies  $\tau(1-q_j^{i,k})$ converges to $0$ as $j$ tends to $\infty$ for all $ i \in [\kappa]$ and $k \in \N$.\\
		
	\noindent	Now choose a subsequence $(j_l)_{l \in \N}$ such that for all $ i \in [\kappa]$
		\begin{align*}
			\tau(1-q^{i,k}_{j_l}) < \frac{\upsilon}{2^{l+k+2}}  \text{ for all } k \in \N.
		\end{align*}
		
		\noindent	For $ i \in [\kappa]$ and $l \in \N$ consider the projections $q^i_l:= \underset{k \in \N  }{ \bigwedge} q^i_{j_l,k}$. Then it follows that 
		\begin{align*}
			\tau(1- q^{i}_l) \leq \frac{\upsilon}{\kappa 2^{l+2}}.
		\end{align*}
		Furthermore, for all $ i \in [\kappa], ~ l \in \N$ and $n \in \{ 1, 2, \cdots, N_{j_l+1} \}$ and $x \in \CM_+$
		\begin{align*}
			A_n(\phi - \phi_{j_l}) (  q_l^i  x  q_l^i ) \leq  \frac{1}{2^{k+ 2j_l}}\norm{x} + \frac{1}{2^{j_l}} \rho(   q_l^i x    q_l^i) \text{ for all } k \in \N.
		\end{align*}
		Therefore, tending $k$ to infinity we have, for all $ i \in [\kappa]$ and $n \in \{ 1, 2, \cdots, N_{j_l+1} \}$ and $x \in \CM_+$
		\begin{align*}
			A_n(\phi - \phi_{j_l}) (  q_l^i  x  q_l^i ) \leq  \frac{1}{2^{j_l}} \rho(   q_l^i x    q_l^i) .
		\end{align*}
		
		\noindent Moreover since $\tau(1-f_j) \to 0$ as $j \to \infty$, we choose projections $(f_{j_l})$ such that 
		\begin{align*}
			\tau(1- f_{j_l}) \leq \frac{\upsilon}{2^{l+2}}.
		\end{align*}
		
		\noindent Finally consider the projection
		\begin{align*}
			e= \Big( \overset{\kappa}{\underset{  i = 1} \bigwedge}~  \underset{l \in \N}{  \bigwedge}  q_l^i  \Big) \bigwedge \Big( \underset{l \in \N}{  \bigwedge} f_{j_l} \Big) \bigwedge q_{s_0}.
		\end{align*}
		Then $\tau(1-e) < \upsilon$ and  for  all $0 \neq x \in e \CM_+ e$, and  suppose $N_{j_l} \leq n, m \leq N_{j_l+1}$, we have
		
		\begin{align*}
			&\frac{\abs{	 \Big(A_n( \phi) - A_m(\phi)\Big) (x)}}{\tau(x)}\\
			=&  \frac{\abs{	 \Big(A_n( \phi)- A_n(\phi_{j_l})  + A_n(\phi_{j_l}) -A_m(\phi_{j_l}) + A_m(\phi_{j_l})- A_m(\phi)\Big) (x)}}{\tau(x)}\\
			\leq &  \frac{\abs{\Big(A_n( \phi)- A_n(\phi_{j_l}) \Big)(x)  } }{\tau(x) }+ \frac{\abs{ \Big(A_n(\phi_{j_l}) -A_m(\phi_{j_l}) \Big) (x) }}{\tau(x)} + \frac{ \abs{ \Big( A_m(\phi_{j_l})- A_m(\phi)\Big) (x)}}{\tau(x)}\\
			\leq &  \frac{\abs{\Big(A_n( \phi- \phi_{j_l}) \Big)(x)  } }{\tau(x) }+ \frac{\abs{ \Big(A_n(\phi_{j_l}) -A_m(\phi_{j_l}) \Big) (x) }}{\tau(x)} + \frac{ \abs{ \Big( A_m(\phi_{j_l}- \phi)\Big) (x)}}{\tau(x)}\\
			\leq &  \Big(\frac{\abs{\Big(A_n( \phi- \phi_{j_l}) \Big)(x)  } }{\rho(x) } + \frac{ \abs{ \Big( A_m(\phi_{j_l}- \phi)\Big) (x)}}{\rho(x)} \Big)\frac{\rho(x)}{\tau(x)} + \frac{\abs{ \Big(A_n(\phi_{j_l}) -A_m(\phi_{j_l}) \Big) (x) }}{\tau(x)}\\
			\leq& s_0  \frac{2}{2^{j_l+1}} + \frac{1}{2^{j_l}} = (s_0 +1) \frac{1}{2^{j_l}} ~( \text{by Eq. } \ref{rho-tau}). 
		\end{align*}	
		
\noindent		Note that as $m,n \to \infty$, $j_l \to \infty$. Therefore,
		\begin{align*}
			\lim_{m,n \to \infty} \sup_{x \in eM_+e, x \neq 0} \frac{\abs{	 \Big(A_n( \phi) - A_m(\phi)\Big) (x)}}{\tau(x)}=0
		\end{align*}

		Therefore, $(A_n(\phi))_{n \in \N}$ is Cauchy in b.a.u.. Therefore, by Proposition \ref{complete wrt bau-functionals}, it follows that $(A_n(\phi))_{n \in \N}$ is convergent in b.a.u., which implies the set $\CC$ is closed in $\CM_*$.
	\end{proof}
	
\noindent 	Finally, we have the following ergodic theorem.
	
	\begin{thm}\label{erg thm-1}
		Let $(\CM, G, \alpha, \rho)$ be a kernel and $\tau$ be a f.n. tracial state on $\CM$. Suppose  $\phi \in \CM_*$, then there exists $\bar{\phi} \in \CM_*$ such that for all $\epsilon>0$ there exists projection $e \in \CM$ with $\tau(1-e)< \epsilon$ and 
		\begin{align*}
			\lim_{n \to \infty} \sup_{x \in eM_+e, x \neq 0} \frac{\abs{(A_n(\phi) - \bar{\phi})(x)}}{\tau(x)}=0.
		\end{align*}
	\end{thm}
	
	\begin{proof}
		The proof follows from Proposition \ref{dense set} and Theorem \ref{Banach principle} in accordance with Remark \ref{dense set-rem1}.
	\end{proof}

	Let $(\CM, G, \alpha, \rho)$ be a kernel and $\tau$ be a f.n. tracial state on $\CM$. Now we can state the Theorem \ref{erg thm-1} in $L^1(\CM, \tau)$. Let $X \in L^1(\CM, \tau)$. Recall Eq. \ref{predual trans} and consider the averages
	\begin{align*}
		A_n(X):= A_{r_n}(X):= \frac{1}{\mu(V_{r_n})} \int_{V_{r_n}} \hat{\alpha}_{s^{-1}}(X) d \mu(s).
	\end{align*}
	
	Similar to Eq. \ref{symm avg}, observe that 
	\begin{align*}
		A_n(X)= \frac{1}{\mu(V_{r_n})} \int_{V_{r_n}} \hat{\alpha_s}(X) d \mu(s).
	\end{align*}

	\begin{thm}\label{erg thm-2}
		Let $(\CM, G, \alpha, \rho)$ be a kernel and $X \in L^1(\CM, \tau)$. Then there exist $\bar{X} \in L^1(\CM, \tau)$ such that $A_n(X)$ converges b.a.u. to $\bar{X}$.
	\end{thm}
	
	\begin{proof}
		The proof follows from Theorem \ref{erg thm-1} and Remark \ref{equiv of bau convg}.
	\end{proof}

	\section{\textbf{Stochastic Ergodic Theorem}}
	
	This section combines our previous results to prove a stochastic ergodic theorem. Throughout this section, as before, we assume that $\CM \subseteq \CB(\h)$ is a von Neumann algebra with a f.n. tracial state  $\tau$ and we consider $L^1(\CM, \tau)$. Further, we also take $G$ as an amenable group metric measure space satisfying the doubling condition.
	Then, we consider a covariant system $( \CM, G, \alpha)$. 
As earlier, for $ X \in L^1(\CM, \tau)$, we consider the following average 
\begin{align*}
A_n(X):= A_{r_n}(X):= \frac{1}{\mu(V_{r_n})} \int_{V_{r_n}} \alpha^*_s(\nu) d \mu(s)=  \frac{1}{\mu(V_{r_n})} \int_{V_{r_n}} \hat{\alpha}_{s^{-1}}(X) d \mu(s).
\end{align*}
Then using the previous results, we established the stochstic ergodic theorem for the average $A_n(\cdot)$. First, we recall the following theorem from \cite{Bik-Dip-neveu}.  
\begin{thm}[Neveu Decomposition] \label{neveu decomposition}
Let $(\CM,  G, \alpha )$ be a covariant system and $\tau$ be a f.n. tracial state on $\CM$. Then there exist two projections $e_1, e_2 \in \CP(\CM)$ such that $e_1 + e_2 = 1$ and $e_1e_2=0$ such that 
\begin{itemize}
	\item[(i)] there exists a $G$-invariant normal state $\rho$ on $\CM$  with support $s(\rho) = e_1$ and
	\item[(ii)] there exists an operator $x_0 \in \CM$ with support $s(x_0)= e_2$ such that $  \frac{1}{\mu(V_{r_n})} \int_{V_{r_n}} \alpha_s(x_0) d \mu(s) \xrightarrow{ n \ri \infty } 0$ in $\norm{ \cdot}$.
\end{itemize}
Further, $s(\rho)$ and $s(x_0)$ are  unique. 
\end{thm}
Now, we collect the following basic facts from  \cite{Bik-Dip-neveu}. 
\begin{enumerate}
	\item $ \alpha_g(e_i) = e_i $ for all $ g \in G$ and $ i =1, 2$.
	\item The restriction of $\alpha$ to the reduced von Neumann algebra $\CM_{e_i}$ ($= e_i \CM e_i$) induces an action by automorphisms, and we denote this induced action by the same notation $\alpha$. Thus, $(\CM_{e_i}, G, \alpha)$ for $i = 1, 2$ becomes a non-commutative dynamical system.

	\item 
	 For $i=1,2$, we also note that for all $g \in G$, the predual transformation  $\hat{\alpha_g}$ defined on $L^1(\CM, \tau)$ satisfies 
	\begin{align*}
	\hat{\alpha_g} (e_i X e_i) = e_i \hat{\alpha_g}(X) e_i \text{ for all } X \in L^1(\CM, \tau) \text{ and for all } g \in G.
	\end{align*}
	As a consequence, the ergodic averages satisfy
	\begin{align*}
	e_i A_n(X) e_i = A_n(e_i X e_i) \text{ for all } X \in L^1(\CM, \tau), n \in \N \text{ and } i=1,2.
	\end{align*}
	
\end{enumerate}
We write $\tau_{e_i}= \frac{1}{\tau(e_i)}\tau|_{e_i \CM e_i}$ for $i = 1, 2$. Now, we have the following theorem.

\begin{thm}\label{conv in e1-e2 corner} Let $(\CM, G, \alpha)$  be a covariant system and 
	$\tau$ be a f.n. tracial state on $\CM$.  Consider the projections $e_1, e_2 \in \CM$ as mentioned above. Then, we have the following results.  
	\begin{enumerate}
		\item[(i)] For all $B \in L^1(\CM_{e_1}, \tau_{e_1})$, there exists $\bar{B} \in L^1(\CM_{e_1}, \tau_{e_1})$ such that $A_n(B)$ converges b.a.u to $\bar{B}$. Moreover,  $A_n(B)$ converges in measure to $\bar{B}$.
		\item[(ii)] For all $B \in L^1(\CM_{e_2}, \tau_{e_2})$, $A_n(B)$ converges to $0$ in measure.
		
	\end{enumerate}
Moreover, for  $X \in L^1(\CM, \tau)$, there exists $Z \in L^1(\CM, \tau)$ such that $A_n(X)$ converges to $Z$ in measure  and we have  $ e_1Z e_1 = Z $ and $ e_2Z e_2 =0$.
\end{thm}
\begin{proof}
Once we have the  Neveu decomposition (Theorem \ref{neveu decomposition}) and Theorem \ref{erg thm-2}, the proof follows verbatim as \cite[see \S 5]{Bik-Dip-neveu}. 
\end{proof}

{\bf Acknowledgments:} The second author is thankful to Prof. Eric Ricard for helpful discussions on martingale convergence.	
		

\newcommand{\etalchar}[1]{$^{#1}$}
\providecommand{\bysame}{\leavevmode\hbox to3em{\hrulefill}\thinspace}
\providecommand{\MR}{\relax\ifhmode\unskip\space\fi MR }
\providecommand{\MRhref}[2]{%
	\href{http://www.ams.org/mathscinet-getitem?mr=#1}{#2}
}
\providecommand{\href}[2]{#2}

\end{document}